\documentclass[review]{elsarticle}

\usepackage{amsmath,amssymb,amsthm,bm,mathtools}
\usepackage{graphicx}
\usepackage{subcaption}
\usepackage{booktabs}
\usepackage{multirow}
\usepackage{algorithm}
\usepackage{algpseudocode}
\usepackage{hyperref}
\usepackage{color}
\usepackage{enumitem}
\usepackage{multicol}

\newtheorem{theorem}{Theorem}

\newcommand{\R}{\mathbb{R}}

\newcommand{\Lop}{\mathcal{L}}

\begin{document}

\begin{frontmatter}

\title{IGA-ODIL: Optimizing DIscretre robust Loss with Isogeometric Analysis to solve forward and inverse problems faster using machine learning tools.}

\author{Maciej Paszy\'nski, Tomasz S\l{}u\.zalec}
\address{Faculty of Computer Science, AGH University of Krakow, Poland}

\begin{abstract}
Physics-informed neural networks (PINNs) formulate the solution of partial differential equations as residual minimization problems over neural network parameterizations. Although highly flexible, optimization of PINNs using modern variants of Stochastic Gradient Descent algorithms is expensive. On the other hand, iterative computation of PINN parameterization using the Gauss-Newton method suffers from convergence  difficulties, dense Jacobian structures, and poor conditioning that limit the effectiveness of second-order optimization methods. In this work, we introduce IGA-ODIL, a spline-based residual minimization framework combining ideas from Optimizing DIscrete Loss (ODIL), robust variational residual minimization, and Isogeometric Analysis (IGA). Instead of neural-network parameterizations of PINNs, the unknown solution is represented by smooth B-spline basis functions, leading to sparse structured Jacobians and efficient Gauss--Newton optimization.
We also derive robust residual formulations based on weighted Gram operators, making the loss function related with the true error. 
The resulting systems inherit locality, sparsity, and approximation-theoretic properties of classical finite element and isogeometric methods while preserving the residual-learning philosophy of scientific machine learning.
The proposed methodology is evaluated on several benchmark problems, including Poisson equations, convection-dominated advection--diffusion equations, Helmholtz problems with highly oscillatory solutions, nonlinear Allen--Cahn equations, and inverse Helmholtz parameter identification. Numerical experiments demonstrate orders-of-magnitude speedups compared with PINNs and CRVPINNs while maintaining high accuracy and robustness.
\end{abstract}


\end{frontmatter}

\section{Introduction}

Scientific machine learning has recently emerged as a rapidly developing research direction at the intersection of numerical analysis, partial differential equations (PDEs), optimization, and machine learning. Among the most influential approaches are Physics-Informed Neural Networks (PINNs), introduced by Raissi, Perdikaris, and Karniadakis~\cite{raissi}. PINNs formulate the numerical solution of PDEs as residual minimization problems in which the unknown solution is represented by a neural network and the governing equations are enforced through collocation-based residual losses.
Numerous extensions and variants have appeared, including Variational PINNs (VPINNs)~\cite{vpinn}, Deep Ritz methods~\cite{deepritz}, adaptive and domain-decomposition formulations, as well as operator-learning approaches such as Fourier Neural Operators~\cite{operatorlearning}. More general overviews of scientific machine learning and physics-informed learning frameworks can be found in~\cite{karniadakis_book}.
Recent applications show
that PINNs are used particularly often in fluid mechanics, where they support
flow reconstruction, inverse parameter identification, and data assimilation
from sparse or noisy measurements \cite{Cai2021}. They have been applied to
wake flows, biomedical flows, supersonic flows, and experimental stratified
flows, demonstrating their potential for combining measurement data with
Navier--Stokes-type models \cite{Cai2021,Zhu2024}.
PINNs have been used for heat transfer in porous media, where they
can estimate temperature fields, heat fluxes, and effective thermal
conductivity without requiring dense labeled data \cite{Xu2023}. They have
also been applied to two-phase flow and heat transfer problems, where the
physics-informed loss helps enforce the governing equations while learning from
simulation data \cite{Jalili2024}. More recently, PINNs have been considered
as components of digital twins, where the goal is to combine sensor data,
reduced-order modeling, and physical constraints for real-time prediction and
uncertainty-aware monitoring \cite{Yang2024}. 

Despite their flexibility, PINNs suffer from several important numerical and theoretical difficulties. Optimization landscapes associated with deep neural-network parameterizations are highly nonconvex and frequently exhibit poor conditioning, flat directions, vanishing gradients, and parameter symmetries~\cite{wang2021understanding}. As a consequence, first-order optimization methods such as stochastic gradient descent (SGD) and Adam~\cite{Adam} often converge slowly and unreliably. Furthermore, second-order methods such as Newton or Gauss–Newton optimization are difficult to apply effectively because the resulting Jacobians are typically large, dense, and ill-conditioned~\cite{nocedal}.

An additional difficulty is that the classical PINN residual loss generally lacks robustness with respect to physically meaningful norms. In many cases, the residual loss does not accurately represent the discretization error measured in Sobolev norms, leading to poor correlation between optimization convergence and physical solution quality~\cite{rvpinn}. This observation motivated the development of Robust Variational PINNs~\cite{rvpinn}, based on the weak residuals of Variational PINNs (VPINNs)~\cite{vpinn}. 
Additionally, to speed up the loss estimation,  Collocation-based Robust Variational Physics-Informed Neural Networks (CRVPINNs)~\cite{crvpinn} have been recently proposed. In particular, the robust loss takes the form
\begin{equation}
 \Phi(u)=R(u)^T G^{-1}R(u),
\end{equation}
where $G$ is a Gram matrix associated with a suitable variational inner product \cite{rvpinn}, and $R(u)$ is a weak residual in the case of VPINNs or a discrete weak residual (computed with Kronecker product delta test functions) in the case of CRVPINNs.
Although robust residual formulations improve the stability of PINN losses, the optimization challenges associated with deep nonlinear parameterizations remain. Neural-network Jacobians remain globally coupled and dense, limiting the scalability and effectiveness of second-order optimization methods.

A complementary research direction was recently proposed by Karnakov, Litvinov, and Koumoutsakos through the Optimizing DIscrete Loss (ODIL) framework~\cite{odil}. The central idea of ODIL is that residual minimization does not fundamentally require neural-network parameterizations. Instead, the unknown solution is represented directly through discrete degrees of freedom optimized using PDE residual losses. ODIL demonstrated that many scientific machine learning ideas can be interpreted as optimization-based PDE discretizations rather than purely neural-network methods.
ODIL interprets PDE solving as optimization over discrete degrees of freedom using residual losses evaluated at collocation points. The resulting formulation preserves many advantages of scientific machine learning, including natural extensions to inverse problems and PDE-constrained optimization.

However, point-based ODIL discretizations are not globally continuous like neural network representations. They also lack several important properties associated with classical spline and finite element approximation spaces. In particular, point-value parameterizations do not naturally provide high-order smoothness, compactly supported basis functions,
and approximation-theoretic properties associated with spline spaces and classical Galerkin discretizations~\cite{quarteroni}.

Isogeometric Analysis, introduced by Hughes, Cottrell, and Bazilevs~\cite{iga}, establishes spline-based discretizations as powerful alternatives to classical finite element methods. 
B-splines and NURBS provide smooth, high-order approximation spaces with excellent approximation properties, sparse local support, and direct compatibility with CAD geometries~\cite{iga,cottrell,bazilevs,sangalli2015smooth,buffa2012electromagnetics}. 
Recent applications of isogeometric analysis (IGA) show that it is widely used
in problems where smoothness, high-order accuracy, and exact geometry
representation are important. In fracture mechanics, IGA has been applied to
phase-field models of brittle, anisotropic, and polycrystalline fracture, where
the higher continuity of spline spaces is advantageous for fourth-order and
gradient-enhanced formulations \cite{Chen2020,Goswami2020,NguyenThanh2020,Greco2024}.
Another active direction is design and topology optimization, where the common
representation of geometry and analysis fields supports design-through-analysis
workflows \cite{Gao2020}. Recent developments also extend IGA to complex CAD
geometries using analysis-suitable T-splines and G-spline surfaces
\cite{Wei2022,Wen2023}. In fluid mechanics and aerodynamics, IGA has been used
for turbulent flows, moving domains, fluid--structure interaction, and
space--time simulations of complex engineering configurations
\cite{Bazilevs2023,Kuraishi2025}. Applications in biomechanics include cardiac
electrophysiology and ventricular mechanics, where smooth patient-specific
geometries and coupled multiphysics models are essential
\cite{Torre2023,Willems2024}. 

We propose an unified framework IGA-ODIL. The proposed methodology replaces neural-network parameterizations with smooth B-spline approximation spaces while preserving the residual-minimization philosophy of scientific machine learning. Specifically, the unknown solution is represented as

\begin{equation}
u_h(x,y)=
\sum_{i=1}^{N_x}
\sum_{j=1}^{N_y}
c_{ij}
B_i^{(p)}(x)
B_j^{(q)}(y),
\label{eq:intro_spline}
\end{equation}
where the trainable variables are spline coefficients rather than neural-network parameters.
The key observation of this work is that spline parameterizations fundamentally alter the structure of residual minimization problems. The Jacobian of spline discretization inherits the sparse locality structure of classical PDE discretizations. Consequently, the Gauss–Newton systems become sparse and structured;
second-order optimization becomes computationally tractable;
least-squares PDE structure reappears;
and locality and scalability are recovered.
In contrast to neural-network parameterizations, spline spaces therefore restore many advantages of classical numerical PDE methods while preserving the optimization-based residual-learning formulation characteristic of scientific machine learning.
The proposed framework additionally incorporates robust variational residual formulations using weighted Gram operators. The resulting method combines robust residual minimization,
 smooth spline approximation,
 sparse Gauss-Newton optimization,
 PDE-constrained inverse optimization,
and scientific machine learning formulations.

The methodology is evaluated on several benchmark problems, including Poisson equations, the Eriksson-Johnson model problem, high-frequency Helmholtz problems,
nonlinear Allen–Cahn equations,
and PDE-constrained inverse problems.
The numerical experiments demonstrate substantial computational advantages compared with PINNs and CRVPINNs, including orders-of-magnitude speedups while maintaining high approximation quality and robustness.

The remainder of the paper is organized as follows. Section~2 reviews robust residual minimization and Gauss--Newton optimization for PINNs. Section~3 reformulates the ODIL method for the robust loss case. Section~4 introduces the IGA-ODIL formulation for robust loss and it derives the corresponding sparse Gauss--Newton systems. Section~5 presents numerical experiments for forward PDE benchmarks, including Poisson, advection-diffusion, Helmholtz, and Allen-Cahn problems. Section~6 considers PDE-constrained inverse problems. We conclude the paper in Section~7. We also include an Appendix A with code references and Appendix B with proof of convergence.

\section{PINNs and Robust Residual Minimization}

\subsection{PINN and CRVPINN formulations}

Consider a PDE operator
\begin{equation}
\Lop u=f
\end{equation}
in a domain $\Omega$.
PINNs approximate the solution using a neural network, 
\begin{equation}
    u_\theta(x,y) = \mathrm{NN}(x,y) = \sigma(A_n \sigma ( \cdots \sigma( 
    A_1
    \begin{bmatrix}
    x\\y
    \end{bmatrix}
    + \beta_1) \cdots) + \beta_n)
\end{equation}
where $A_i$ are layer weights,  $\sigma$ are nonlinear activation functions, and $\beta_i$ are biases. In general,  $\theta\in\R^P$ denotes the trainable parameters~\cite{raissi}.
The PINN method is summarized in Figure 1.
The residual evaluated at collocation points $\{{\bf x}_k\}_{k=1}^M$ is
\begin{equation}
 R_k(\theta)=\Lop u_\theta({\bf x}_k)-f({\bf x}_k).
\end{equation}
The classical PINN loss 
\begin{equation}
 \mathcal{J}(\theta)=\frac12\sum_{k=1}^M |R_k(\theta)|^2.
\end{equation}
can be interpreted as a nonlinear least-squares optimization problem over neural-network parameterizations~\cite{deepritz,nocedal}.
The classical way of solving the PINN problem is through some modern variants of the Stochastic Gradient Descent method. The ADAM algorithm \cite{Adam}, which replaces the local gradient with the weighted average gradient from several previous iterations, is one of the most popular choices. The method is summarized in Algorithm 1.
\begin{algorithm}
\caption{Physics-Informed Neural Network (PINN)}
\begin{algorithmic}[1]

\State Construct neural network approximation
$
u_\theta({\bf x})
$

\State Select collocation points
$
\{{\bf x}_k\}_{k=1}^M
$

\State Initialize network parameters
$
\theta^{(0)}
$

\While{not converged}

    \State Evaluate residuals at collocation points
    $
    R_k(\theta^{(n)})
    =
    \Lop u_{\theta^{(n)}}({\bf x}_k)
    -
    f({\bf x}_k),
    \qquad
    k=1,\dots,M
    $

    \State Form residual vector
    $
    R(\theta^{(n)})
    =
    \left[
    R_1(\theta^{(n)}),
    \dots,
    R_M(\theta^{(n)})
    \right]^T
    $

    \State Compute PINN loss
    $
    \mathcal{J}(\theta^{(n)})
    =
    \frac12
    \sum_{k=1}^M
    |R_k(\theta^{(n)})|^2
    +
    \mathcal{L}_{BC}
    $

    \State Compute gradients
    $
    \nabla_\theta
    \mathcal{J}(\theta^{(n)})
    $
    using backpropagation and automatic differentiation

    \State Update neural-network parameters
    $
    \theta^{(n+1)}
    =
    \theta^{(n)}
    -
    \eta
    \nabla_\theta
    \mathcal{J}(\theta^{(n)})
    $
    using Adam, L-BFGS, or related optimizers

\EndWhile

\end{algorithmic}
\end{algorithm}    
PINNs solved with SGD/ADAM methods are generally computationally much more demanding than Finite Element Method/Isogeometric Analysis solvers.

\begin{figure}
\includegraphics[width=\textwidth]{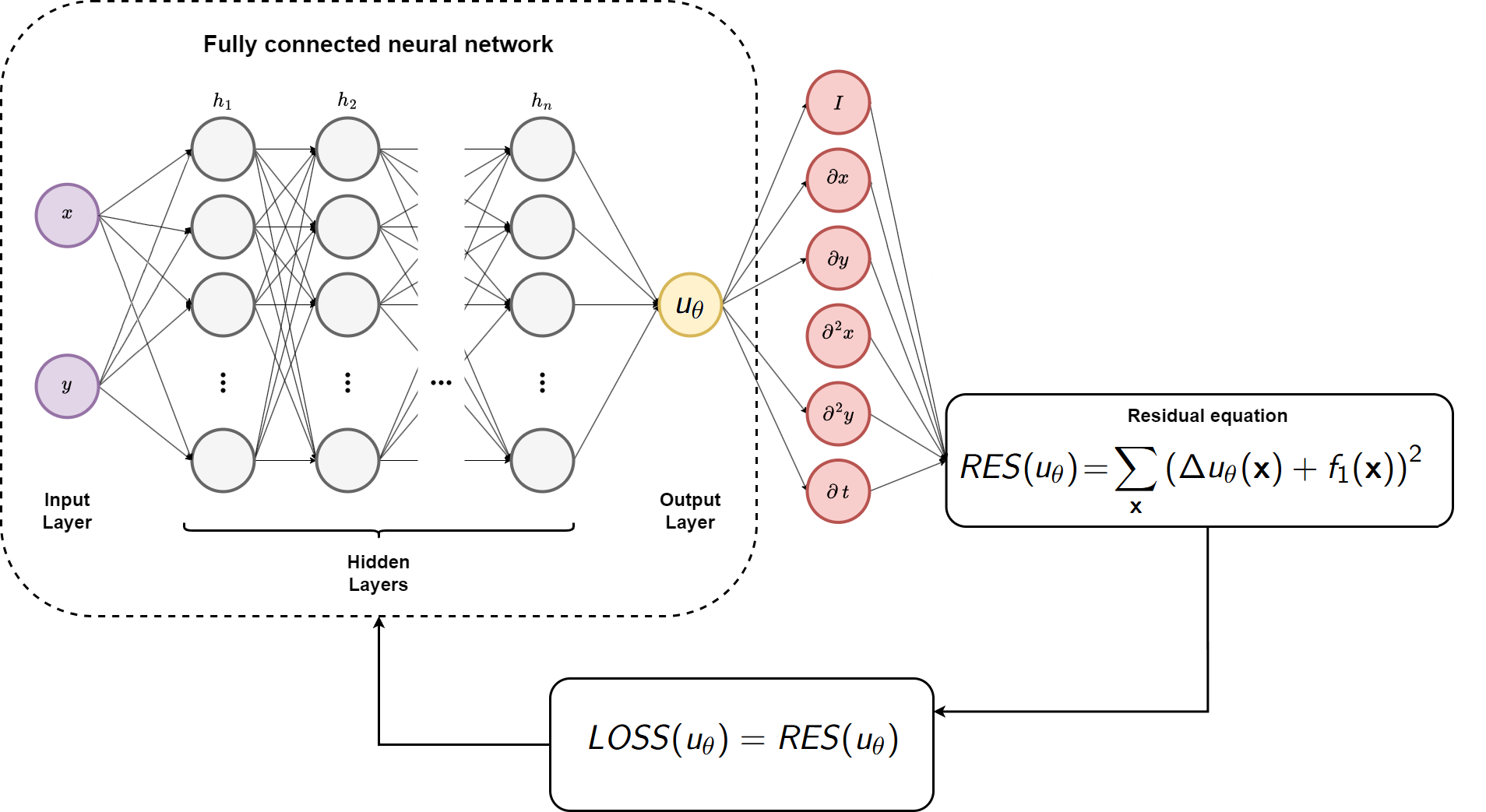}
\caption{PINN (Physics Informed Neural Networks)}
\label{fig1}
\end{figure}

\subsection{Gauss--Newton Optimization for PINNs}

An alternative way to compute $\theta$ coefficients of neural network approximation may be the application of second order optimization methods like the Gauss-Newton.
We can derive the Gauss--Newton method for solving the residual minimization problem when the solution is represented by NN.
Let $u_\theta({\bf x})$ denote NN parameterized by $\theta \in \mathbb{R}^P$
(composition of affine transformations and nonlinear activation functions). 
The PDE residual is given by
\begin{equation}
R_k(\theta) = \mathcal{L} u_\theta({\bf x}_k) - f({\bf x}_k),
\end{equation}
and evaluated at a set of collocation points $\{{\bf x}_k\}_{k=1}^M$:
\begin{equation}
R(\theta) = \big( R_1(\theta), \dots, R_M(\theta) \big)^T=  \big( R({\bf x}_1;\theta), \dots, R({\bf x}_M;\theta) \big)^T \in \mathbb{R}^M.
\end{equation}
The objective is to minimize the nonlinear least-squares functional
\begin{equation}
\mathcal{J}(\theta) = \frac{1}{2} \|R(\theta)\|_2^2.
\end{equation}
The Jacobian matrix is defined as
\begin{equation}
J(\theta) = \frac{\partial R(\theta)}{\partial \theta} \in \mathbb{R}^{M \times P}.
\end{equation}
(here $M=$ \# collocation points, $P=$ \# NN parameters)
Using the chain rule, each row of $J(\theta)$ (single point) is given by
\begin{equation}
\frac{\partial R({\bf x}_k;\theta)}{\partial \theta}
= \frac{\partial}{\partial \theta} \left( \mathcal{L} u_\theta({\bf x}_k) \right)
= \mathcal{L} \left( \frac{\partial u_\theta({\bf x}_k)}{\partial \theta} \right).
\end{equation}
The term $\frac{\partial u_\theta({\bf x})}{\partial \theta}$ can be computed efficiently using backpropagation.
At an iterate $\theta^{(k)}$, we linearize the residual
\begin{equation}
R(\theta^{(k)} + \delta) \approx R(\theta^{(k)}) + J(\theta^{(k)}) \delta.
\end{equation}
The Gauss-Newton step is obtained by minimizing the quadratic model
\begin{equation}
\min_{\delta} \; \| R(\theta^{(k)}) + J(\theta^{(k)}) \delta \|_2^2.
\end{equation}
This leads to the normal equations
\begin{equation}
J(\theta^{(k)})^T J(\theta^{(k)}) \, \delta^{(k)}
= - J(\theta^{(k)})^T R(\theta^{(k)}).
\label{eq:gn_nn}
\end{equation}
The parameters are updated as follows:
\begin{equation}
\theta^{(k+1)} = \theta^{(k)} + \delta^{(k)}.
\end{equation}
The method is summarized in Algorithm 2.
\begin{algorithm}
\caption{Gauss--Newton Training of Physics-Informed Neural Networks}
\begin{algorithmic}[1]

\State Construct neural network approximation $u_\theta({\bf x})$
\State Select collocation points $\{{\bf x}_k\}$
\State Initialize network parameters $\theta^{(0)}$

\While{not converged}

    \State Evaluate residual vector
    $
    R(\theta^{(k)})
    =
    \left[
    R({\bf x}_1;\theta^{(k)}),
    \dots,
    R({\bf x}_{N_r};\theta^{(k)})
    \right]^T
    $

    \State Assemble Jacobian matrix
    $
    J(\theta^{(k)})
    =
    \frac{\partial R(\theta^{(k)})}{\partial \theta}
    $
    using automatic differentiation

    \State Solve the Gauss--Newton system
    $
    \left(
    J^T J
    \right)\delta^{(k)}
    =
    -J^T R
    $

    \State Update neural network parameters
    $
    \theta^{(k+1)}
    =
    \theta^{(k)}
    +
    \delta^{(k)}
    $

\EndWhile


\end{algorithmic}
\end{algorithm}
Such formulations correspond to nonlinear least-squares optimization methods~\cite{nocedal}.
However, the Jacobian $J(\theta)$ is typically large and dense, 
and the resulting systems are large and ill-conditioned.
Neural-network Jacobians are typically dense, highly nonlinear, and poorly conditioned~\cite{wang2021understanding}. In practice, the resulting systems are often singular or numerically unstable, limiting the applicability of second-order optimization methods for large-scale PINN formulations.

\subsection{Gauss--Newton Optimization for CRVPINNs}

Classical residual losses employed by PINNs are generally not equivalent to physically meaningful energy norms~\cite{rvpinn}. Robust variational residual minimization introduces weighted losses of the form
\begin{equation}
 \Phi(u)=\frac12 R(u)^T G^{-1}R(u),
\end{equation}
where $G$ is a symmetric positive definite Gram matrix associated with a variational inner product~\cite{rvpinn,crvpinn}.
Following the discrete weak formulation described in \cite{crvpinn} we have 
\begin{equation}
\frac{1}{\mu}\sqrt{ \Phi(u)}
\le
\|u-u_{exact}\|_{\nabla,h}
\le
\frac{1}{\alpha}\sqrt{ \Phi(u)}.
\end{equation}
Thus, the robust loss becomes equivalent to the discretization of the true error measured in suitable discrete Sobolev norms~\cite{crvpinn}. Classical way of solving CRVPINNs is using Stochastic Gradient Descent or ADAM \cite{Adam} method.
We can also derive the Gauss-Newton method for solving the residual minimization problem when the solution is represented by CRVPINN.
Let $u_\theta(x)$ denote a neural network parameterized by
$\theta \in \mathbb{R}^P$.
The PDE residual is defined as
\begin{equation}
R(x;\theta) = \mathcal{L}u_\theta(x)-f(x),
\end{equation}
and evaluated at collocation points $\{x_k\}_{k=1}^M$:
\begin{equation}
R(\theta)
=
\big(
R(x_1;\theta),
\dots,
R(x_M;\theta)
\big)^T
\in
\mathbb{R}^M.
\end{equation}

Unlike standard PINNs, CRVPINNs employ weighted or norm-equivalent residual minimization to improve stability and convergence properties.
The objective is to minimize the robust least-squares functional
\begin{equation}
\mathcal{J}(\theta)
=
\frac12
R(\theta)^T
G^{-1}
R(\theta),
\end{equation}
where $G$ is a Gram matrix associated with a suitable Sobolev or energy norm.

The Jacobian matrix is defined as
\begin{equation}
J(\theta)
=
\frac{\partial R(\theta)}{\partial \theta}
\in
\mathbb{R}^{M \times P}.
\end{equation}

($M=$ number of collocation points,\quad
$P=$ number of NN parameters)

Using the chain rule,
\begin{equation}
\frac{\partial R(x_k;\theta)}{\partial \theta}
=
\frac{\partial}{\partial \theta}
\left(
\mathcal{L}u_\theta(x_k)
\right)
=
\mathcal{L}
\left(
\frac{\partial u_\theta(x_k)}{\partial \theta}
\right).
\end{equation}
At iteration $\theta^{(k)}$, we linearize the residual:
\begin{equation}
R(\theta^{(k)}+\delta)
\approx
R(\theta^{(k)})
+
J(\theta^{(k)})\delta.
\end{equation}

The Gauss-Newton step minimizes the quadratic model
\begin{equation}
\min_\delta
\;
\|
G^{-1/2}
(
R + J\delta
)
\|_2^2.
\end{equation}

This leads to the weighted normal equations:
\begin{equation}
J^T G^{-1} J \, \delta^{(k)}
=
-
J^T G^{-1} R.
\end{equation}

The parameters are updated by
\begin{equation}
\theta^{(k+1)}
=
\theta^{(k)}
+
\delta^{(k)}.
\end{equation}

The method is summarized in Algorithm 3.
\begin{algorithm}
\caption{Gauss--Newton Optimization for CRVPINNs}
\begin{algorithmic}[1]

\State Construct neural network approximation $u_\theta({\bf x})$

\State Select collocation points $\{{\bf x}_k\}$

\State Assemble Gram matrix $G$

\State Compute LU factorization once:
$
G = LU
$

\State Initialize network parameters $\theta^{(0)}$

\While{not converged}

    \State Evaluate residual vector
    $
    R(\theta^{(k)})
    =
    \left[
    R({\bf x}_1;\theta^{(k)}),
    \dots,
    R({\bf x}_{N_r};\theta^{(k)})
    \right]^T
    $

    \State Assemble the Jacobian matrix
    $
    J(\theta^{(k)})
    =
    \frac{\partial R(\theta^{(k)})}{\partial \theta}
    \in \mathbb{R}^{N_r \times P}
    $
    using automatic differentiation

    \State Solve
    $
    LU\, y_R = R
    $

    \State Compute weighted Jacobian
    $
    LU\, Y_J = J
    $
    where each column of $Y_J$ is obtained by solving
    $
    LU\, y_j = J_j
    $
    with $J_j$ denoting the $j$-th column of $J$

    \State Assemble Gauss--Newton Hessian approximation
    $
    H = J^T Y_J
    $

    \State Assemble right-hand side
    $
    g = J^T y_R
    $

    \State Solve the linearized system
    $
    H \delta^{(k)} = -g
    $

    \State Update neural network parameters
    $
    \theta^{(k+1)}
    =
    \theta^{(k)}
    +
    \delta^{(k)}
    $

\EndWhile


\end{algorithmic}
\end{algorithm}
However, unlike spline discretizations, the Jacobian
$J(\theta)$ is typically large and dense,
and the weighted Hessian approximation
$J^T G^{-1}J$

\section{ODIL and Fast Robust Residual Minimization}

\subsection{ODIL formulation}

\begin{figure}
\centering
\includegraphics[width=\textwidth]{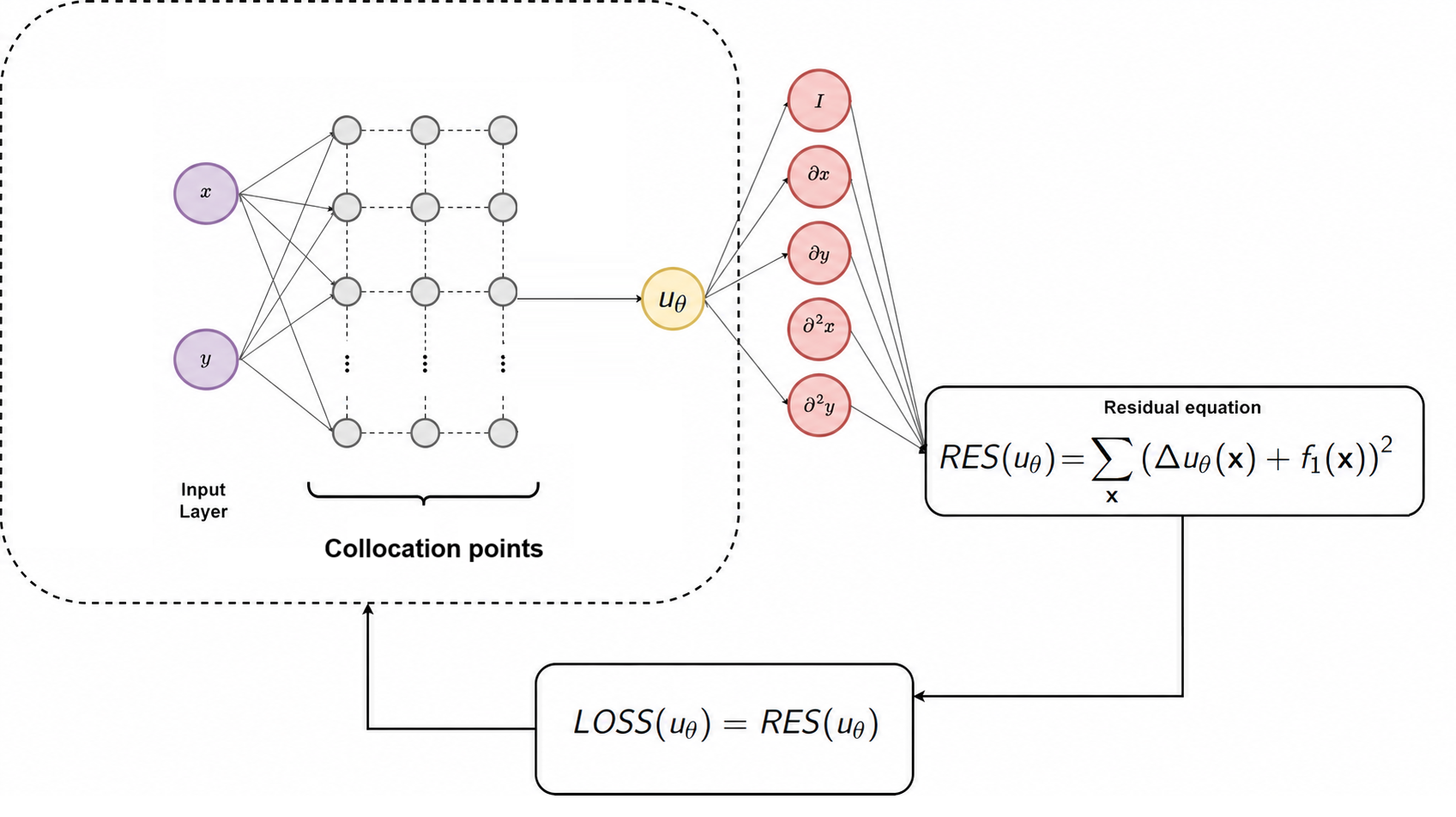}
\caption{ODIL (Optimizing DIscrete Loss)}
\end{figure}

In the ODIL (Optimizing DIscrete Loss) method, the unknown solution is represented directly by discrete values at collocation points.
For the Poisson equation,
we introduce a structured set of collocation points
$
\{(x_i,y_j)\}_{i=1,j=1}^{N_x,N_y},
$
and define the unknown discrete solution values
$
u_{i,j} \approx u(x_i,y_j).
$
The trainable variables are, therefore, the nodal unknowns
$
u \in \mathbb{R}^{N_x\times N_y}.
$
The PDE residual is evaluated directly at collocation points.
For example, using a standard five-point discretization of the Laplacian,

\begin{eqnarray}
R_{i,j}(u)
=
\frac{
u_{i+1,j}
-
2u_{i,j}
+
u_{i-1,j}
}{h_x^2}
+
\frac{
u_{i,j+1}
-
2u_{i,j}
+
u_{i,j-1}
}{h_y^2}
+
f(x_i,y_j).
\end{eqnarray}
The residual vector is obtained by stacking all local residuals:
$
R(u)
=
\left[
R_{1,1}(u),
\dots,
R_{N_x,N_y}(u)
\right]^T.
$
The ODIL loss functional is defined as the discrete residual minimization problem

\begin{equation}
LOSS(u)
=
\|R(u)\|^2
=
\sum_{i,j} R_{i,j}(u)^2.
\end{equation}

To improve robustness and discretization invariance, we introduce the weighted residual functional

\begin{equation}
\Phi(u)
=
\frac12
R(u)^T
G^{-1}
R(u),
\end{equation}

where $G$ denotes the Gram matrix associated with the residual representation.

\subsection{Gauss--Newton Optimization for ODIL}

Using the first-order linearization
$
R(u+\delta)
\approx
R + J\delta,
$
where
$
J(u)
=
\frac{\partial R(u)}{\partial u},
$
the robust functional becomes

\begin{equation}
\Phi(u+\delta)
=
\frac12
(R+J\delta)^T
G^{-1}
(R+J\delta).
\end{equation}

Expanding the quadratic form yields

\begin{equation}
\Phi(u+\delta)
=
\frac12
\left(
R^T G^{-1} R
+
2\delta^T J^T G^{-1} R
+
\delta^T J^T G^{-1} J \delta
\right).
\end{equation}

Differentiating with respect to $\delta$ gives

\begin{equation}
\nabla_\delta \Phi
=
J^T G^{-1} R
+
J^T G^{-1} J \delta.
\end{equation}

Setting the gradient to zero produces the weighted Gauss--Newton equations

\begin{equation}
(J^T G^{-1} J)\delta
=
-
J^T G^{-1} R.
\label{eq:GN_ODIL}
\end{equation}

For linear PDEs, the residual has the affine form $R(u)=Au-b$, 
which implies $J=A.$ Consequently, the Gauss--Newton method reduces to the solution of a weighted least-squares discretization.
To avoid explicit formation of $G^{-1}$, the Gram matrix is factorized once $
G = LU$.
Applications of $G^{-1}$ are then replaced by triangular solves.
The resulting ODIL method is summarized in Algorithm 4.

\begin{algorithm}
\caption{Gauss--Newton Optimization for ODIL}
\begin{algorithmic}[1]

\State Construct collocation points $\{{\bf x}_{i,j}\}$

\State Initialize discrete solution values $u^{(0)}$

\State Assemble Gram matrix $G$

\State Compute LU factorization once:
$
G = LU
$

\While{not converged}

    \State Evaluate residual vector
    $
    R(u^{(k)})
    $

    \State Assemble Jacobian matrix
    $
    J(u^{(k)})
    =
    \frac{\partial R(u^{(k)})}{\partial u}
    $

    \State Solve
    $
    LU\, y_R = R
    $

    \State Compute weighted Jacobian
    $
    LU\, Y_J = J
    $
    where each column of $Y_J$ is obtained by solving
    $
    LU\, y_j = J_j
    $
    with $J_j$ denoting the $j$-th column of $J$

    \State Assemble Gauss--Newton Hessian approximation
    $
    H = J^T Y_J
    $

    \State Assemble right-hand side
    $
    g = J^T y_R
    $

    \State Solve the linearized system
    $
    H\delta^{(k)}
    =
    -g
    $

    \State Update solution
    $
    u^{(k+1)}
    =
    u^{(k)}
    +
    \delta^{(k)}
    $

\EndWhile


\end{algorithmic}
\end{algorithm}
\begin{figure}
\includegraphics[width=0.9\textwidth]{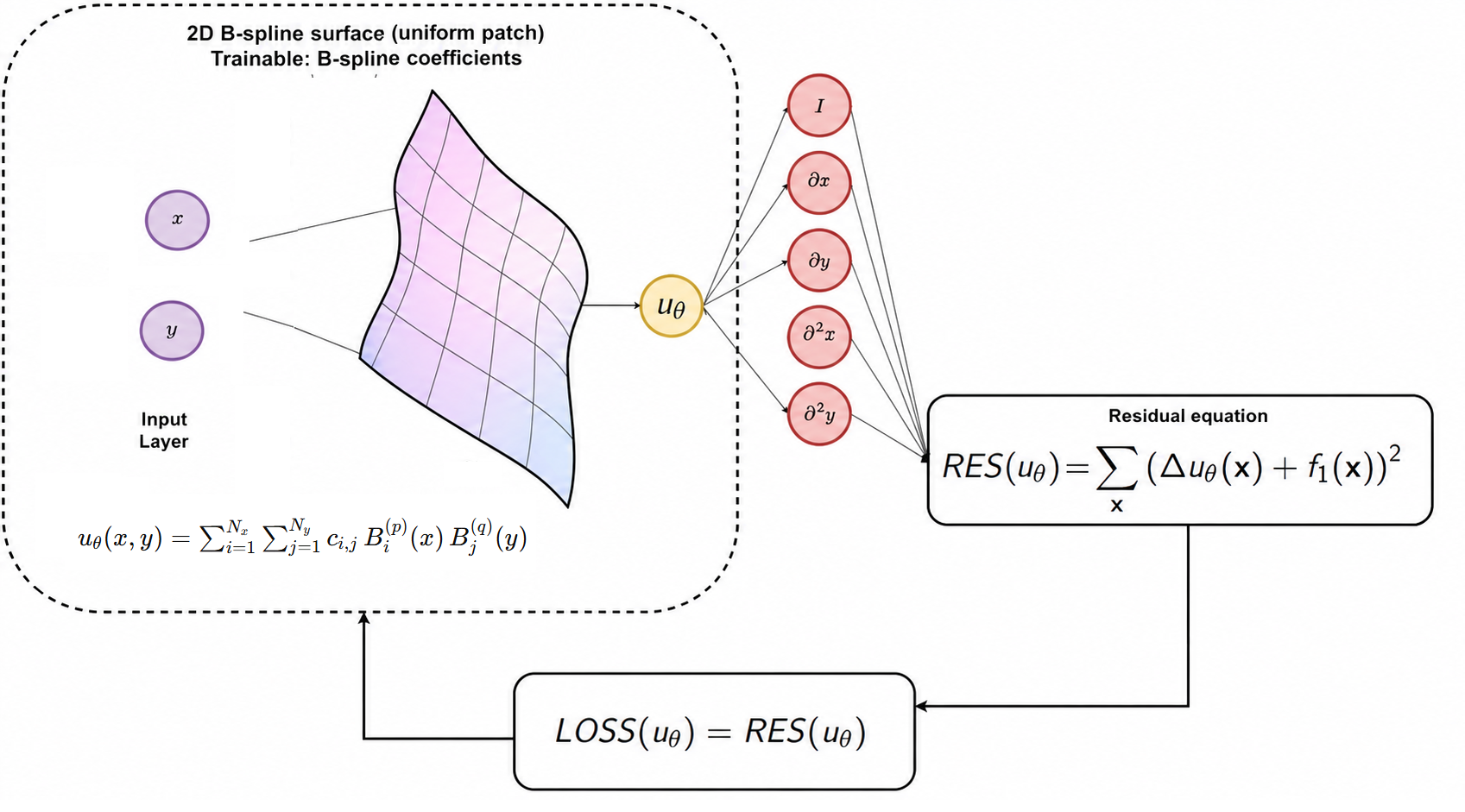}
\caption{IGA-ODIL (Isogeometric Analysis by Optimizing DIscrete Loss)}
\end{figure}
\section{IGA-ODIL and fast Robust Residual Minimization}

\subsection{IGA-ODIL formulation}

In the IGA-ODIL method, the unknown solution is represented using tensor-product B-splines. For example, for the Poisson equation,

\begin{eqnarray}
u_c(x,y) = \sum_{i=1}^{N_x} \sum_{j=1}^{N_y} c_{i,j}\, B_i^{(p)}(x)\, B_j^{(q)}(y), \; \; 
 c \in {\cal M}^{N_x \times N_y} \\
R(u_c)=\sum_{\hat{x},\hat{y}}\left(\frac{\partial^2 u_{c}(\hat{x},\hat{y})}{\partial x^2}+\frac{\partial^2 u_{c}(\hat{x},\hat{y})}{\partial y^2}+f(\hat{x},\hat{y})\right)^2 \\
=\sum_{\hat{x},\hat{y}}\left(
\sum_{i=1}^{N_x} \sum_{j=1}^{N_y} c_{i,j} \left[
\frac{\partial^2 B_i^{(p)}(\hat{x})}{\partial x^2}B_j^{(q)}(\hat{y})+B_i^{(p)}(\hat{x})\frac{\partial^2 B_j^{(q)}(\hat{y})}{\partial^2 y}\right] + f(\hat{x},\hat{y})\right)^2
\end{eqnarray}

The trainable variables are the spline coefficients $c$.
In general, the residual evaluated at collocation points $(\hat{x}_k,\hat{y}_k)$ is
\begin{equation}
R_k(c)=\Lop u_c(\hat{x}_k,\hat{y}_k)-f(\hat{x}_k,\hat{y}_k).
\end{equation}

The robust residual functional is given by

\begin{equation}
\Phi(u)
=
\frac12
R(u)^T G^{-1} R(u).
\end{equation}

\subsection{Gauss-Newton Optimization for IGA-ODIL}

Using the linearization $
R(u+\delta)
\approx
R + J\delta,$ the robust functional 

\begin{equation}
\Phi(u+\delta)
=
\frac12
(R+J\delta)^T
G^{-1}
(R+J\delta).
\end{equation}
\begin{equation}
\Phi(u+\delta)
=
\frac12
\left(
R^T G^{-1} R
+
2\delta^T J^T G^{-1} R
+
\delta^T J^T G^{-1} J \delta
\right).
\end{equation}

Differentiating with respect to $\delta$ gives
$
\nabla_\delta \Phi
=
J^T G^{-1} R
+
J^T G^{-1} J \delta.
$
Setting the gradient to zero yields the weighted normal equations:
\begin{equation}
(J^T G^{-1} J)\delta
=
-
J^T G^{-1} R \label{eq:GN}
\end{equation}
To solve~\eqref{eq:GN}, we require the Jacobian of the residual with respect to the coefficients,
\begin{equation}
J(c) = \frac{\partial R(c)}{\partial c}.
\end{equation}
For linear PDEs the residual has the affine form
$R(c)=Ac-b$, 
which implies $J=A$.
Consequently, Gauss–Newton becomes equivalent to solving a least-squares PDE discretization.
The overall method for solving IGA-ODIL is summarized in Algorithm 5.

\begin{algorithm}
\caption{Gauss--Newton Optimization for IGA--ODIL}
\begin{algorithmic}[1]

\State Construct tensor-product B-spline basis functions
$
\{\phi_{i,j}(x,y)\}
=
\{B_i^{(p)}(x)B_j^{(q)}(y)\}
$

\State Select collocation points $\{{\bf x}_k\}$

\State Initialize spline coefficients $c^{(0)}$

\State Assemble Gram matrix $G$

\State Compute LU factorization once:
$
G = LU
$

\While{not converged}

    \State Evaluate spline solution
    $
    u_h^{(k)}(x,y)
    =
    \sum_{i=1}^{N_x}
    \sum_{j=1}^{N_y}
    c_{i,j}^{(k)}
    B_i^{(p)}(x)
    B_j^{(q)}(y)
    $

    \State Evaluate residual vector
    $
    R(c^{(k)})
    $

    \State Assemble Jacobian matrix
    $
    J(c^{(k)})
    =
    \frac{\partial R(c^{(k)})}{\partial c}
    $

    \State Solve
    $
    LU\, y_R = R
    $

    \State Compute weighted Jacobian
    $
    LU\, Y_J = J
    $
    where each column of $Y_J$ is obtained from
    $
    LU\, y_j = J_j
    $
    with $J_j$ denoting the $j$-th column of $J$

    \State Assemble Gauss--Newton Hessian approximation
    $
    H = J^T Y_J
    $

    \State Assemble right-hand side
    $
    g = J^T y_R
    $

    \State Solve the linearized system
    $
    H\delta^{(k)}
    =
    -g
    $

    \State Update spline coefficients
    $
    c^{(k+1)}
    =
    c^{(k)}
    +
    \delta^{(k)}
    $

\EndWhile


\end{algorithmic}
\end{algorithm}

\section{Numerical Experiments}

In this section, we compare the PINN, CRVPINN, ODIL, and IGA-ODIL methods on a number of benchmarks.

\subsection{Poisson Problem}

We first consider the Poisson equation
\begin{equation}
-\Delta u=f
\end{equation}
with manufactured solution
\begin{equation}
 u(x,y)=\sin(2\pi x)\sin(2\pi y).
\end{equation}

The corresponding forcing term is
\begin{equation}
f(x,y)=8\pi^2\sin(2\pi x)\sin(2\pi y).
\end{equation}

Figure~\ref{fig:pinnloss} illustrates the discrepancy between the classical PINN residual loss and the true $H^1$ error.

\begin{figure}[ht]
\centering
\includegraphics[width=0.65\textwidth]{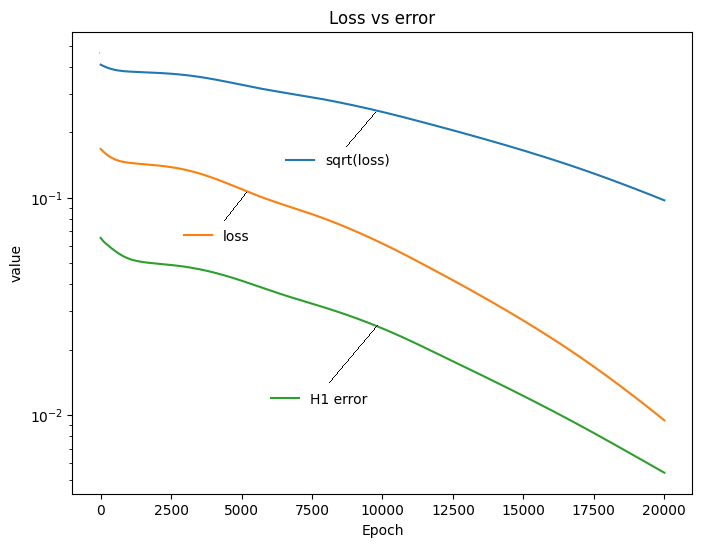}
\caption{PINN residual loss versus true error for the Poisson benchmark.}
\label{fig:pinnloss}
\end{figure}

The robust CRVPINN formulation using the robust loss and discrete weak formulations substantially improves convergence robustness, see Figure \ref{fig:crvpinn}.
The training time using the ADAM algorithm with $100\times 100$ collocation points and 2 fully connected neural network layers is 6 min and 51 sec on Google Compute Engine GPU.
Replacing the ADAM method with the Gauss-Newton formulation with Neural Network parameterization results in lack of convergence

{\tt Residual size: 1600}

{\tt Parameter size: 1801}

{\tt Epoch    0 | Loss = 1.481357e+03 | time = 22.62s}

{\tt Residual size: 1600}

{\tt Parameter size: 1801}

{\tt \_LinAlgError: torch.linalg.solve: The solver failed because the input matrix is singular.}

\begin{figure}[ht]
\centering
\includegraphics[width=0.48\textwidth]{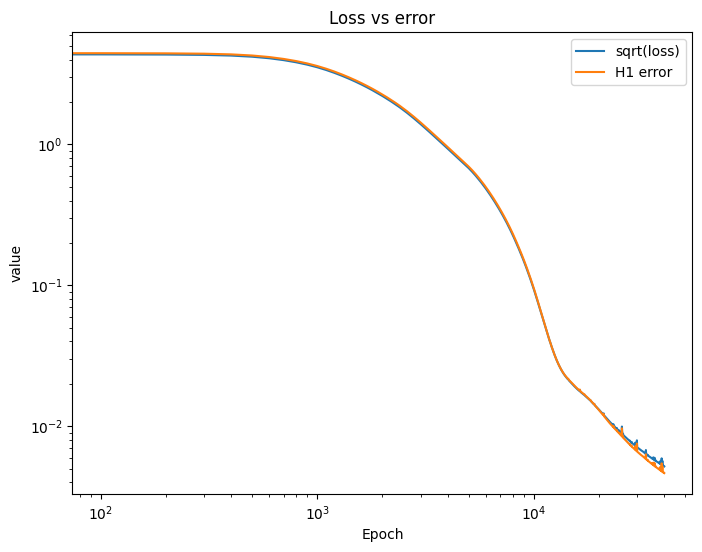}
\includegraphics[width=0.48\textwidth]{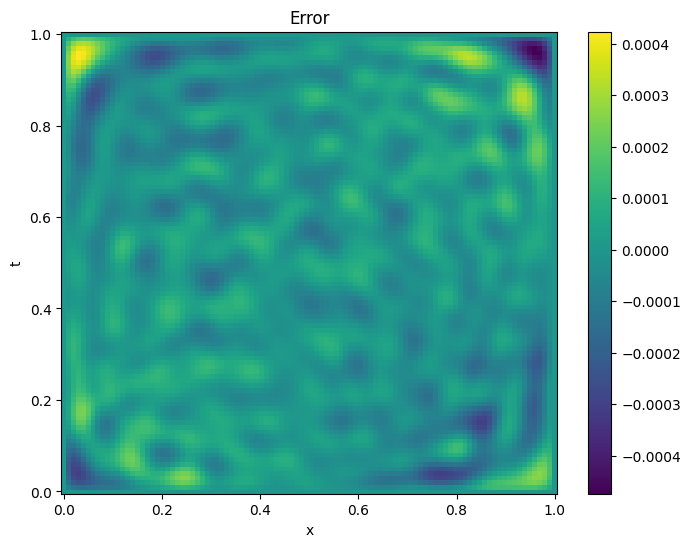}
\caption{Robust CRVPINN convergence and error distribution.}
\label{fig:crvpinn}
\end{figure}

We first switch to the ODIL method, where 
the neural network is replaced by a vector of coefficients $u_\theta = \{ u_{i,j} \}$ and
\begin{eqnarray}
R(u_\theta)=\sum_{x_{i,j}}\left(\frac{\partial u_{i,j}}{\partial x}+\frac{\partial u_{i,j}}{\partial y}+f_1(x_{i,j})\right)^2.  \end{eqnarray}
Running ADAM optimizer with ODIL discretization using $100\times 100$ collocation points takes 60,000 iterations, see Figure \ref{fig:odil_conv}. It requires 90 sec of training in Google Compute Engine GPU, see Figure \ref{fig:odil_sol}.
Solving the ODIL method with Gauss-Newton based on a robust residual iterative solver takes 14.73 sec on Google Compute Engine GPU (28 times faster than CRVPINN). The ODIL loss is not robust and it is not related with the true error. The obtain $L^2$ norm accuracy is 0.00046.
\begin{figure}
\includegraphics[width=0.7\textwidth]{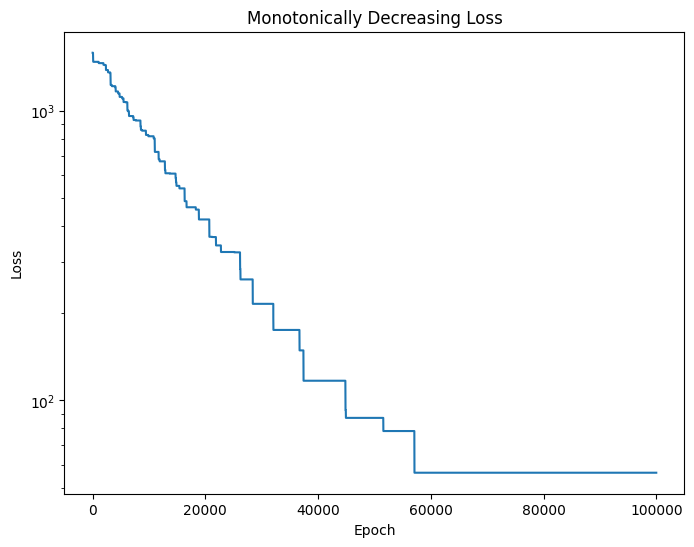}
\caption{Convergence of ODIL training with ADAM method}
\label{fig:odil_conv}
\end{figure}
\begin{figure}
\includegraphics[width=0.49\textwidth]{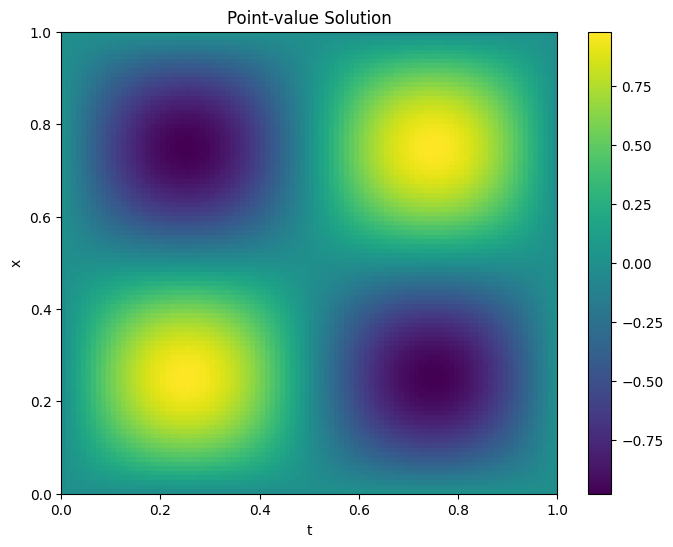}\includegraphics[width=0.49\textwidth]{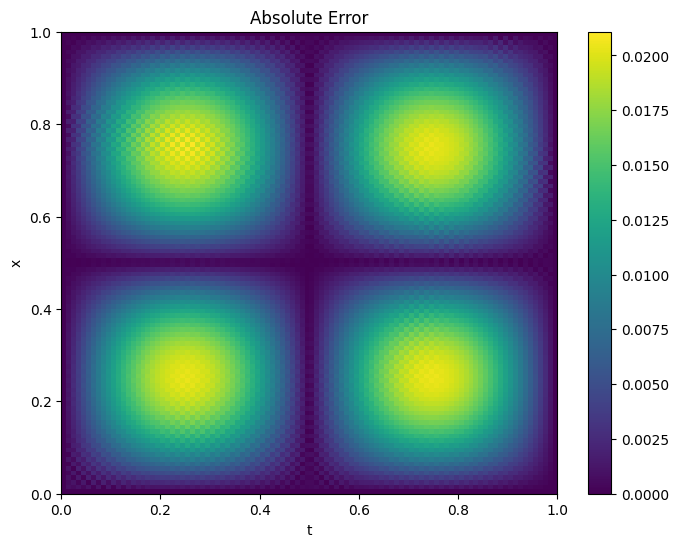} \caption{Solution obtained with ODIL method}
\label{fig:odil_sol}
\end{figure}

Finally, we employ the IGA-ODIL approximation. IGA-ODIL replaces neural-network parameterizations with spline coefficients.
The IGA-ODIL system is solved using the Gauss-Newton method. We employ MATLAB implementation, see Figure \ref{fig:igaodil}.
\begin{figure}
\includegraphics[width=\textwidth]{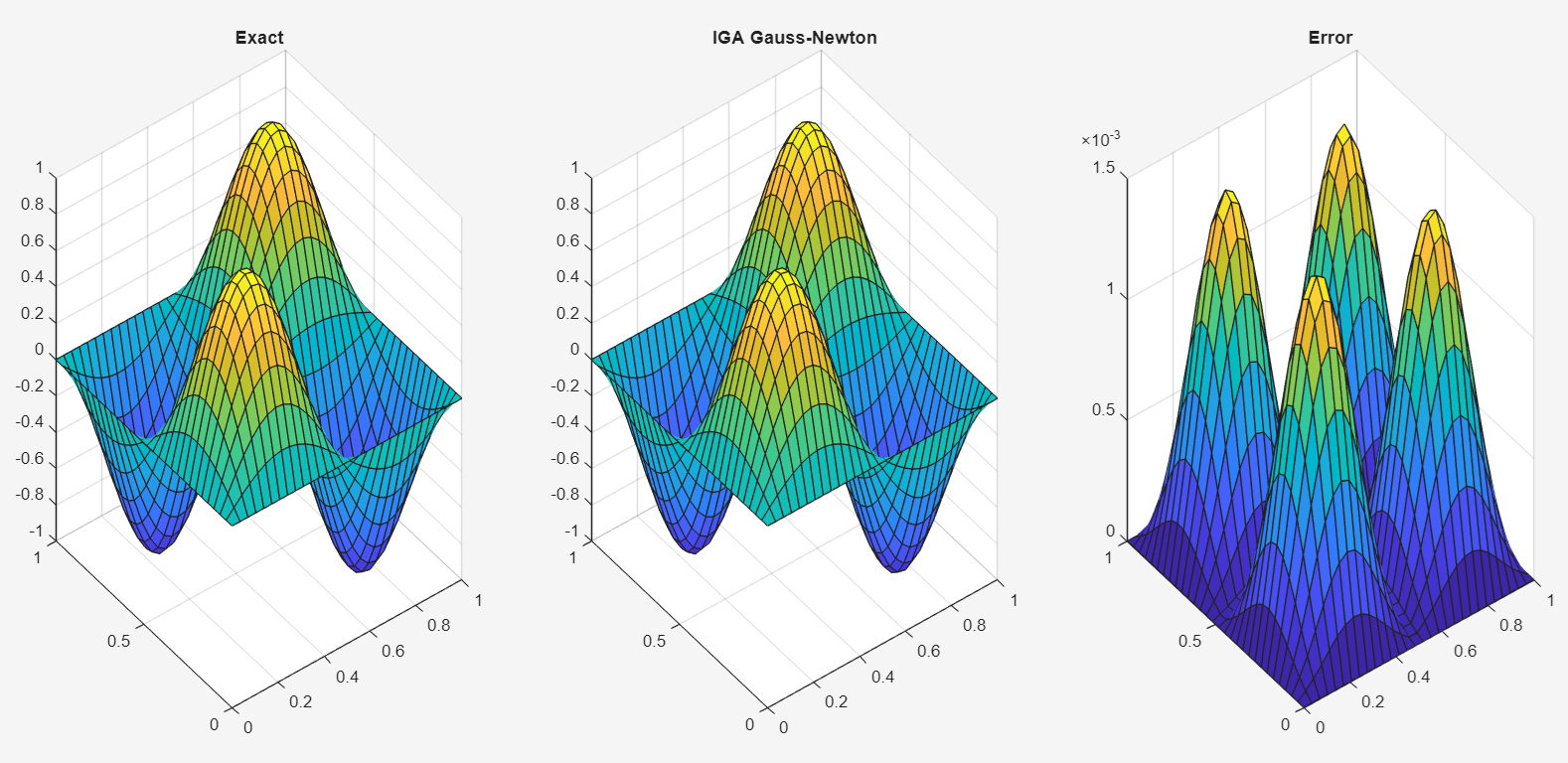}
\caption{IGA-ODIL solution and error for the Poisson problem.}
\label{fig:igaodil}
\end{figure}
We use 30$\times$ 30 elements, cubic B-splines, and 30$\times$30 collocation points. 
The resulting $L^2$ norm error is  6.468394e-04.
The Gauss-Newton method takes 1 iteration and requires 0.04854 sec on MATLAB with a single core.
More detailed evaluation of IGA-ODIL for different number of elements $n_x=n_y$, collocation points $N_c$ and order $p$ of B-splines which are $C^{p-1}$ continous is presented in Table 1 and Figure \ref{fig:num1}. In general, the accuracy increases with order of B-splines, provided there are enough collocation points. This is consistent with Theorem 1 presented in Appendix B.
The IGA-ODIL is three orders of magnitude faster than CRVPINN.

\begin{table}
\centering
\caption{Results for the 2D Poisson problem solved with sparse IGA--ODIL.}
\label{tab:poisson_iga_odil}
\begin{tabular}{ccccc}
\hline
$p$ & $n_x=n_y$ & $N_c$ & Solver time [s] & $L^2$ error \\
\hline
1 & 10 & 30 & $4.618\times 10^{-4}$ & $5.000\times 10^{-1}$ \\
1 & 20 & 30 & $3.922\times 10^{-4}$ & $5.000\times 10^{-1}$ \\
1 & 30 & 30 & $3.450\times 10^{-4}$ & -- \\
2 & 10 & 30 & $5.532\times 10^{-4}$ & $5.187\times 10^{-3}$ \\
2 & 20 & 30 & $3.001\times 10^{-3}$ & $2.700\times 10^{-3}$ \\
2 & 30 & 30 & $3.266\times 10^{-2}$ & $9.704\times 10^{-3}$ \\
3 & 10 & 30 & $8.202\times 10^{-4}$ & $6.324\times 10^{-4}$ \\
3 & 20 & 30 & $3.166\times 10^{-2}$ & $1.162\times 10^{-4}$ \\
3 & 30 & 30 & $4.854\times 10^{-3}$ & $6.468\times 10^{-4}$ \\
\hline
\end{tabular}
\end{table}

\begin{figure}
\includegraphics[width=0.49\textwidth]{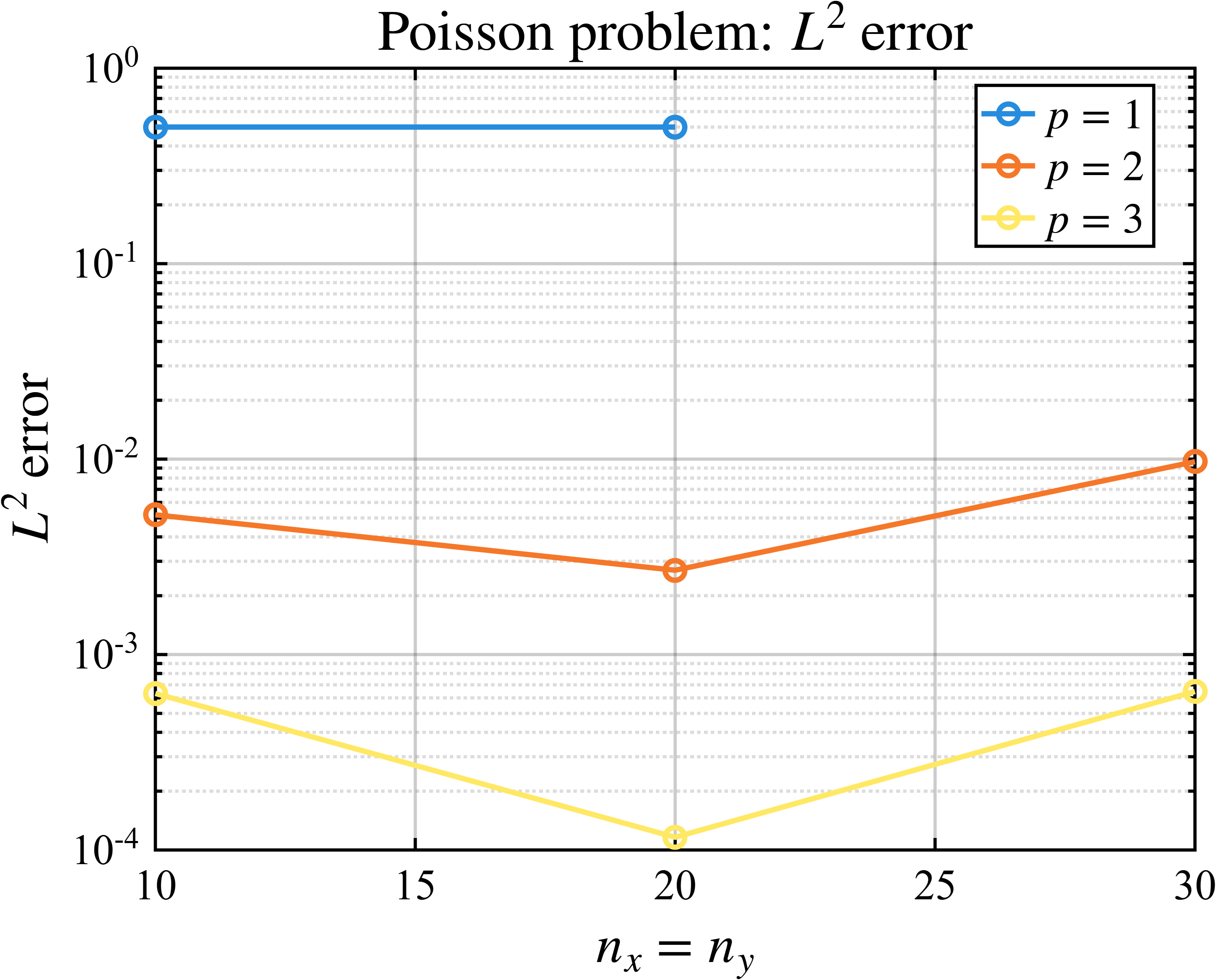}\includegraphics[width=0.49\textwidth]{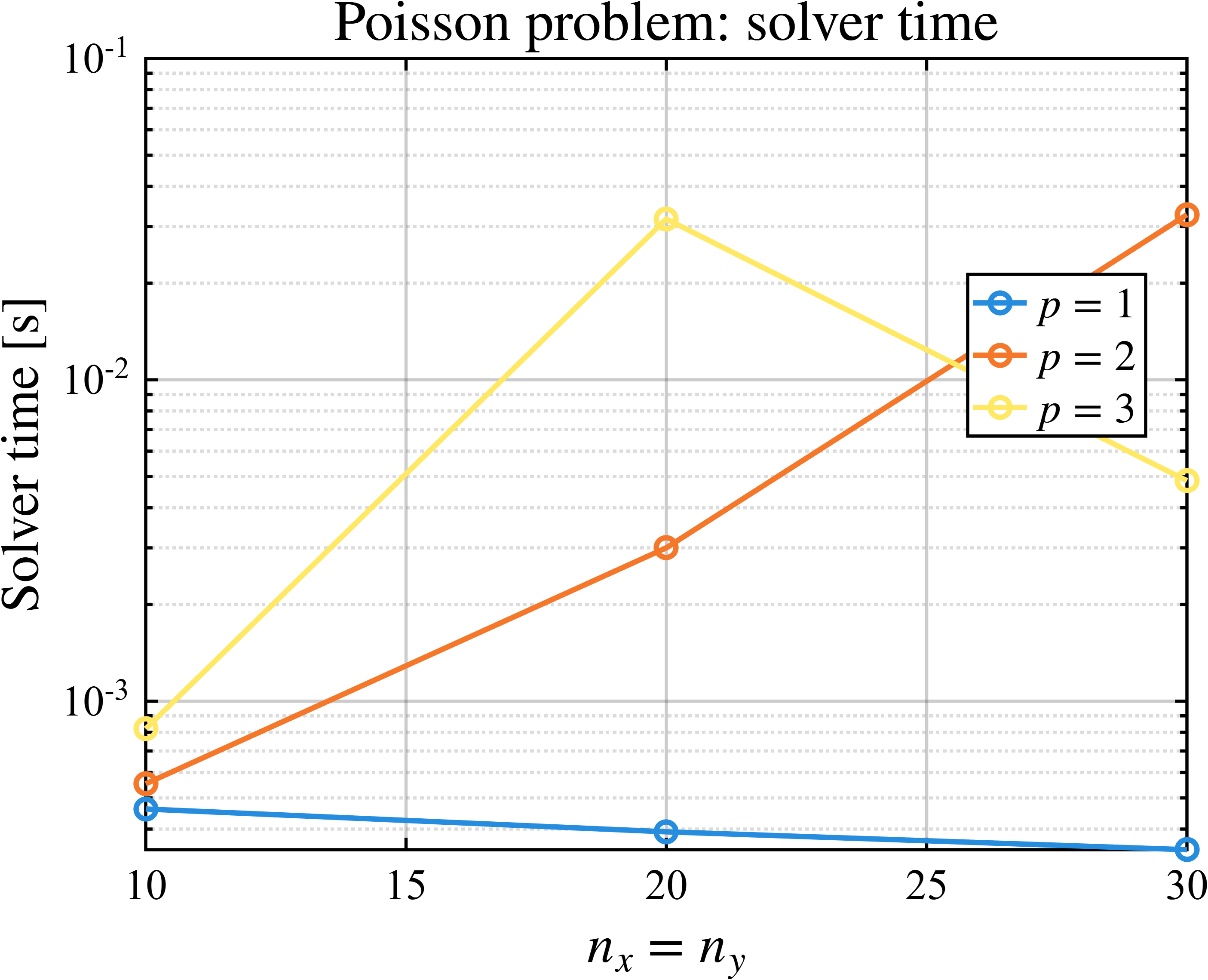}
\caption{Accuracy and timing of IGA-ODIL solution for the Poisson problem for B-splines of order $p=1,2,3$.}
\label{fig:num1}
\end{figure}

\subsection{Eriksson--Johnson Problem}

To assess performance for convection-dominated regimes, we consider the Eriksson–Johnson advection–diffusion problem:
\begin{equation}
-\varepsilon \Delta u + \boldsymbol{\beta} \cdot \nabla u = 0, \quad (x,y) \in \Omega,
\end{equation}
where $\varepsilon > 0$ is the diffusion coefficient and $\boldsymbol{\beta} = (1,0)^T$ is the advection field.
Boundary conditions are prescribed as:
\begin{align}
u(0,y) &= \sin(\pi y), && \text{(inflow)}, \\
u(x,y) &= 0, && \text{on remaining boundaries}.
\end{align}

For small values of $\varepsilon$, the solution exhibits boundary layers near the outflow boundary $x=1$, making this a challenging benchmark for numerical methods.
The problem exhibits the exact solution
\begin{eqnarray}
\label{eq:exact3}
u_{exact}(x, y) = \frac{(e^{(r_1 (x-1))} - e^{(r_2 (x-1))})}{ (e^{(-r_1)} - e^{(-r_2)})}  \sin(\pi  y),  \notag \\ r_1 = \frac{(1 + \sqrt{(1 + 4\epsilon^2\pi^2)})}{ (2\epsilon)}, r_2 = \frac{(1 - \sqrt{(1 + 4\epsilon^2\pi^2)})}{ (2\epsilon)}.
\end{eqnarray}

Applying PINN for the Eriksson-Johnson problem for $\epsilon=0.1$ results in a solution presented in Figure \ref{fig:EJ}. 
The PINN loss is not robust.
Upgrading to the robust loss as proposed by CRVPINN method, using 100$\times 100$ collocation points, two layers with 100 neurons, results in a solution presented in Figure \ref{fig:EJ2}.
Total solution time on Google Compute Engine GPU is 3 min 31 sec.

\begin{figure}
\includegraphics[width=0.45\textwidth]{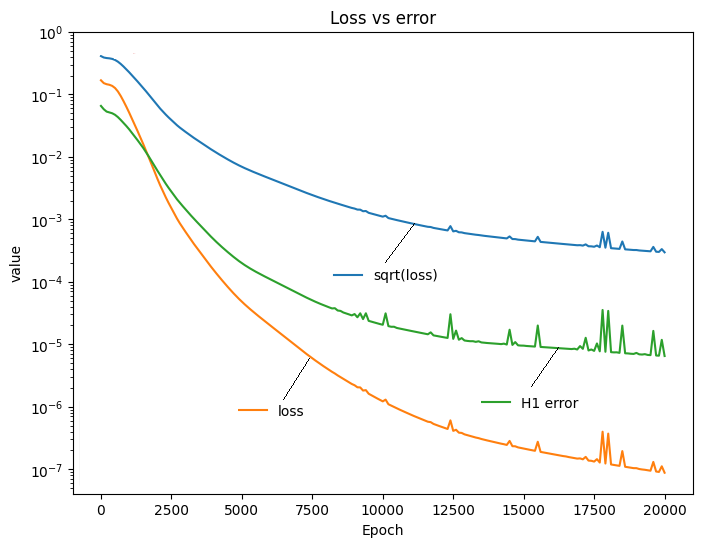}\includegraphics[width=0.45\textwidth]{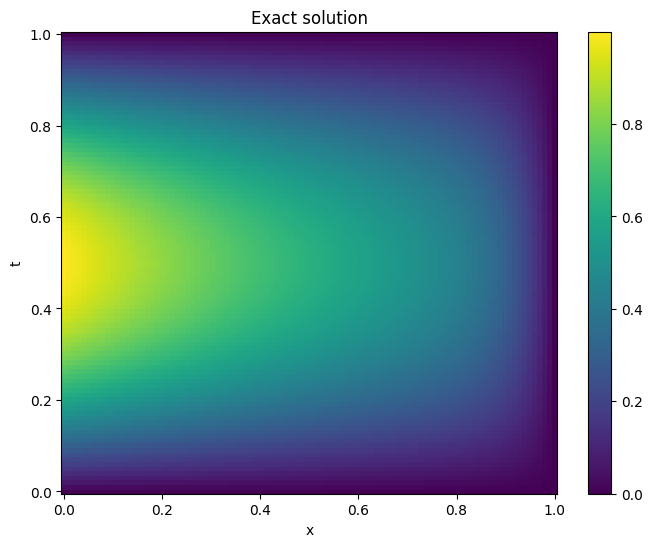}
    \caption{Advection-diffusion problem $\left\{\tiny
 \begin{array}{rl}
 \beta \cdot \nabla u - \epsilon \Delta u = 0 & \hbox{in }\Omega, \\
 u=g & \hbox{over }\partial\Omega\,.
 \end{array}\right.$}
    \label{fig:EJ}
\end{figure}

\begin{figure}
    \includegraphics[width=0.4\textwidth]{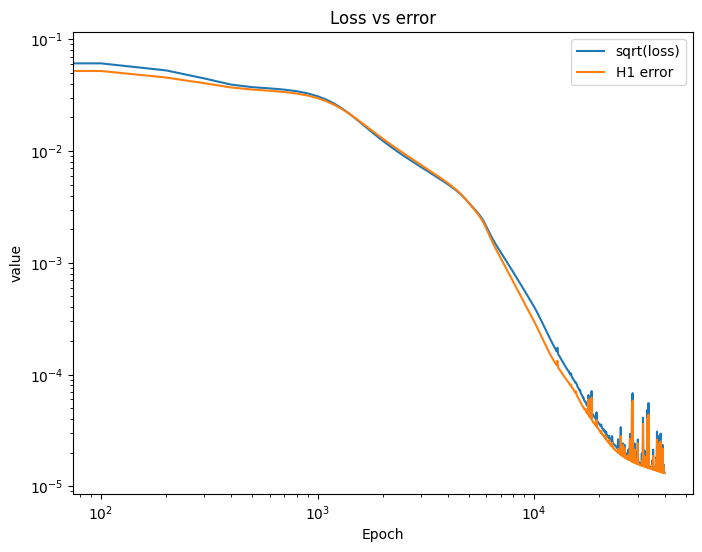}\includegraphics[width=0.4\textwidth]{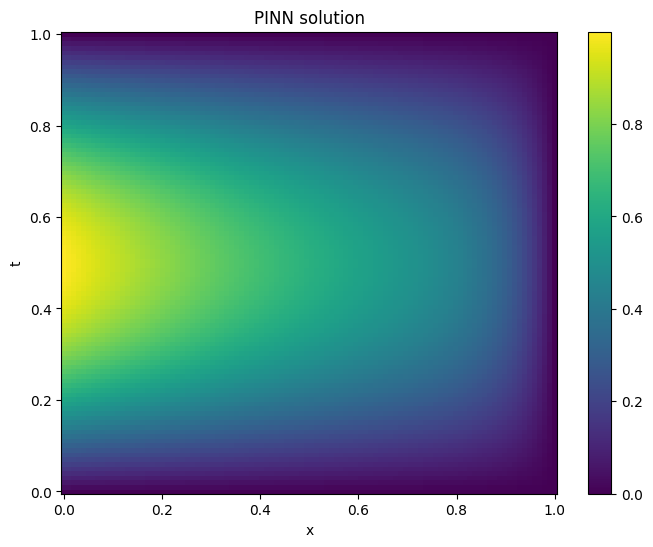}
\caption{\small Advection-diffusion problem {\tiny $\left\{
 \begin{array}{rl}
 \beta \cdot \nabla u - \epsilon \Delta u = 0 & \hbox{in }\Omega, \\
 u=g & \hbox{over }\partial\Omega\,.
 \end{array}\right.$}}
    \label{fig:EJ2}
\end{figure}


Detailed evaluation of IGA-ODIL for $\epsilon=0.001$, for different number of elements $n_x=n_y$, collocation points $N_c$ and order $p$ of B-splines which are $C^{p-1}$ continous is presented in Table 2 and Figure \ref{fig:num2}. In general, the accuracy increases with order of B-splines, provided there are enough collocation points. This is consistent with theoretical estimate of Theorem 1 from Appendix B.
The IGA-ODIL is two orders of magnitude faster than CRVPINN.
\begin{table}[ht]
\centering
\caption{Results for the Eriksson--Johnson problem with $\varepsilon=0.001$.}
\label{tab:ej_eps_0p001}
\begin{tabular}{ccccc}
\hline
$p$ & $n_x=n_y$ & $N_c$ & Solver time [s] & $L^2$ error \\
\hline
1 & 100 & 500 & $1.267\times 10^{-1}$ & $4.032\times 10^{-1}$ \\
1 & 200 & 500 & $3.534\times 10^{-1}$ & $4.031\times 10^{-1}$ \\
1 & 300 & 500 & $8.727\times 10^{-1}$ & $4.031\times 10^{-1}$ \\
2 & 100 & 500 & $3.364\times 10^{-1}$ & $4.004\times 10^{-1}$ \\
2 & 200 & 500 & $7.384\times 10^{-1}$ & $3.978\times 10^{-1}$ \\
2 & 300 & 500 & $1.936\times 10^{0}$ & $2.922\times 10^{-1}$ \\
3 & 100 & 500 & $5.351\times 10^{-1}$ & $3.973\times 10^{-1}$ \\
3 & 200 & 500 & $1.396\times 10^{0}$ & $3.651\times 10^{-1}$ \\
3 & 300 & 500 & $2.993\times 10^{0}$ & $1.844\times 10^{-2}$ \\
\hline
\end{tabular}
\end{table}

\begin{figure}
\includegraphics[width=0.49\textwidth]{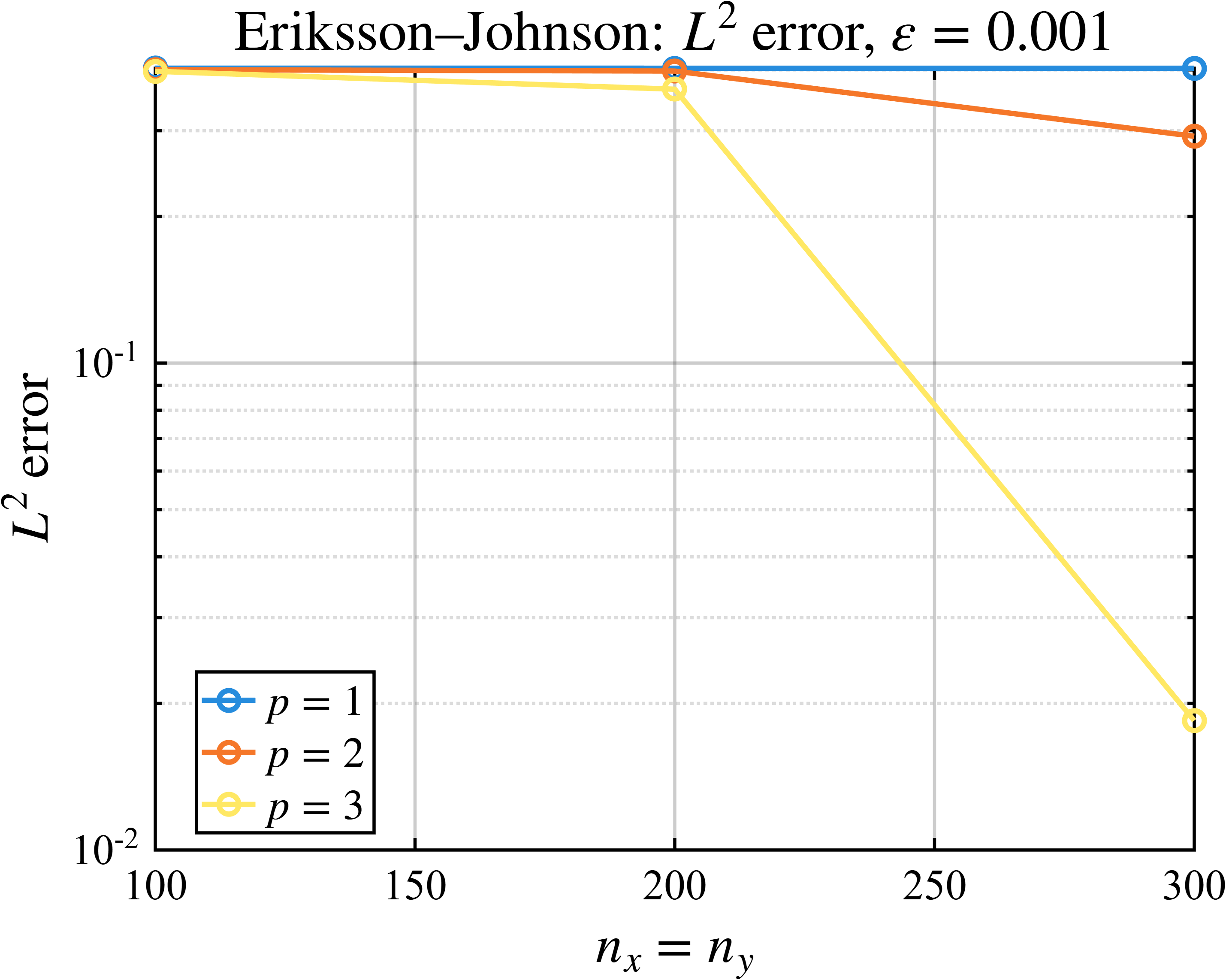}\includegraphics[width=0.49\textwidth]{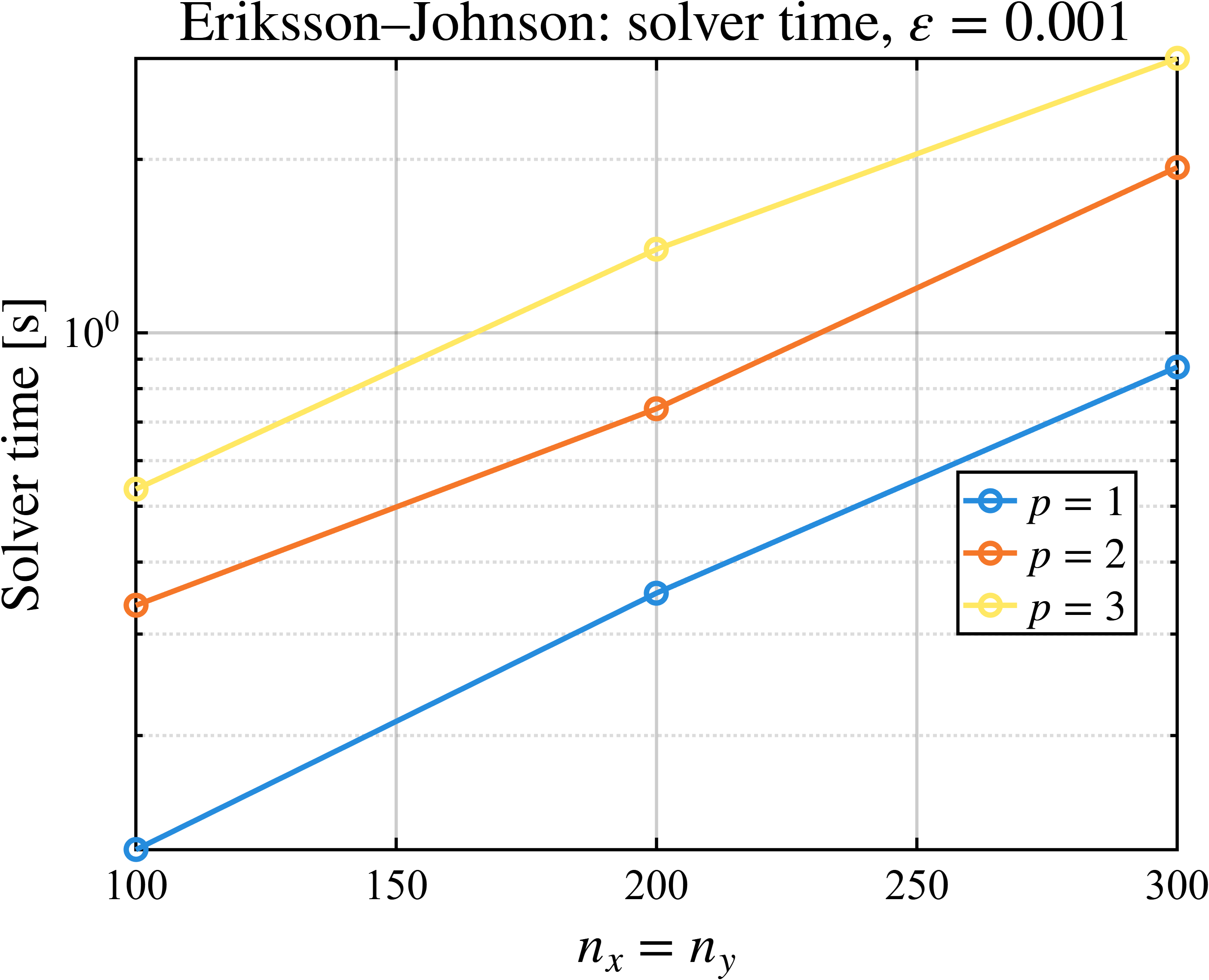}
\caption{Accuracy and timing of IGA-ODIL solution for the Eriksson-Johnson problem for $\epsilon=0.001$ for B-splines of order $p=1,2,3$.}
\label{fig:num2}
\end{figure}

\subsection{Helmholtz Problem}

We consider the Helmholtz problem:
\begin{equation}
-\Delta u + \alpha u = f, \quad (x,y) \in \Omega,
\end{equation}
with homogeneous Dirichlet boundary conditions
\begin{equation}
u(x,y) = 0, \quad (x,y) \in \partial \Omega.
\end{equation}

A manufactured solution is chosen as
\begin{equation}
u(x,y) = \sin(\kappa\pi x)\sin(\kappa\pi y),
\end{equation}
where $\kappa \in \mathbb{N}$ controls the frequency. The corresponding forcing term is
\begin{equation}
f(x,y) = \left( 2(\kappa\pi)^2 + \alpha \right) \sin(\kappa\pi x)\sin(\kappa\pi y).
\end{equation}

This problem is particularly challenging for large values of $\kappa$ due to oscillatory behavior and the well-known pollution effect.
We set up the problem with $\kappa=40$. PINN, CRVPINN, and ODIL methods fail to solve this difficult Helmholtz problem for $\kappa=40$.
IGA-ODIL successfully resolves oscillatory solutions with sparse Gauss-Newton optimization.
We employ $80\times 80$ elements, cubic B-splines, and $160\times 160$ collocation points. The resulting accuracy in $L^2$ norm is 1.467537e-02.  The solution takes 0.15 sec on a MATLAB single core CPU.

\begin{figure}[ht]
\centering
\includegraphics[width=0.8\textwidth]{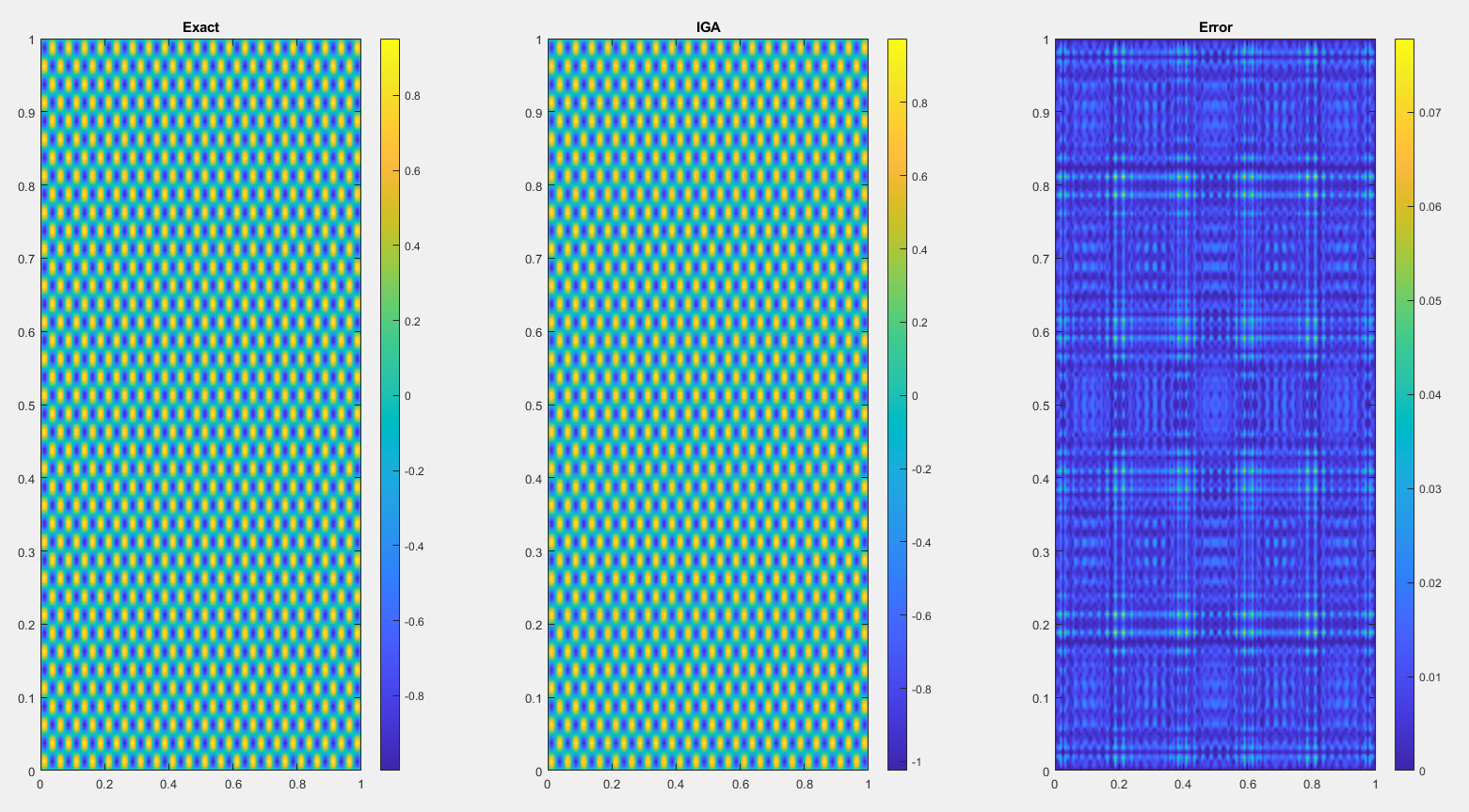}
\caption{IGA-ODIL solution for the Helmholtz problem with $\kappa=40$.}
\end{figure}

\subsection{Three-Dimensional Helmholtz Problem}

We further consider a three-dimensional Helmholtz benchmark, see Figure \ref{figh3d}.
We solve the problem for $\kappa=10$ using $20 \times 20\times 20$ elements, cubic B-splines, $40\times 40 \times 40$ collocation points, $\kappa=10$, and the Gauss-Newton method.
The resulting accuracy in $L^2$ norm is 1.902649e-02. The problem takes 7.82 sec on a MATLAB single core CPU.

\begin{figure}[ht]
\centering
\includegraphics[width=0.65\textwidth]{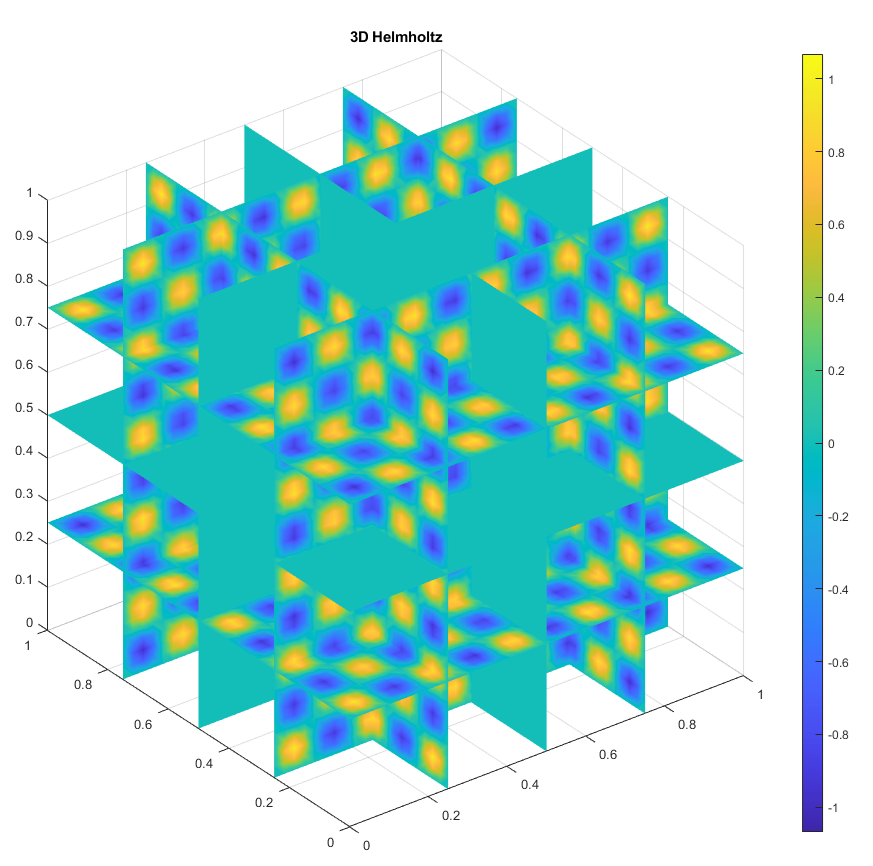}
\caption{Three-dimensional Helmholtz solution obtained with IGA-ODIL.}
\label{figh3d}
\end{figure}

\subsection{Allen--Cahn Equation}

We consider the steady Allen--Cahn equation:
\begin{equation}
-\varepsilon^2 \Delta u + (u^3 - u) = f, \quad (x,y) \in \Omega,
\end{equation}
with Dirichlet boundary conditions
\begin{equation}
u(x,y) = g(x,y), \quad (x,y) \in \partial \Omega.
\end{equation}

Using a manufactured solution
\begin{equation}
u(x,y) = \tanh\left(\frac{x+y-1}{\sqrt{2}\epsilon}\right),
\end{equation}
the forcing term is given by
\begin{equation}
f(x,y) = \tanh\left(\frac{x+y-1}{\sqrt{2}\epsilon}\right) - \tanh^3\left(\frac{x+y-1}{\sqrt{2}\epsilon}\right).
\end{equation}

The Allen--Cahn equation represents a nonlinear reaction--diffusion system exhibiting bistability and sharp interface layers for small values of $\varepsilon$, making it a suitable benchmark for nonlinear solvers.
Detailed evaluation of IGA-ODIL for Allen-Cahn problem with $\epsilon=0.01$, for different number of elements $n_x=n_y$, collocation points $N_c$ and order $p$ of B-splines which are $C^{p-1}$ continous is presented in Table 3 and Figure \ref{fig:num3}. In general, the accuracy increases with number of elements, provided there are enough collocation points. 

\begin{figure}[ht]
\centering
\includegraphics[width=0.9\textwidth]{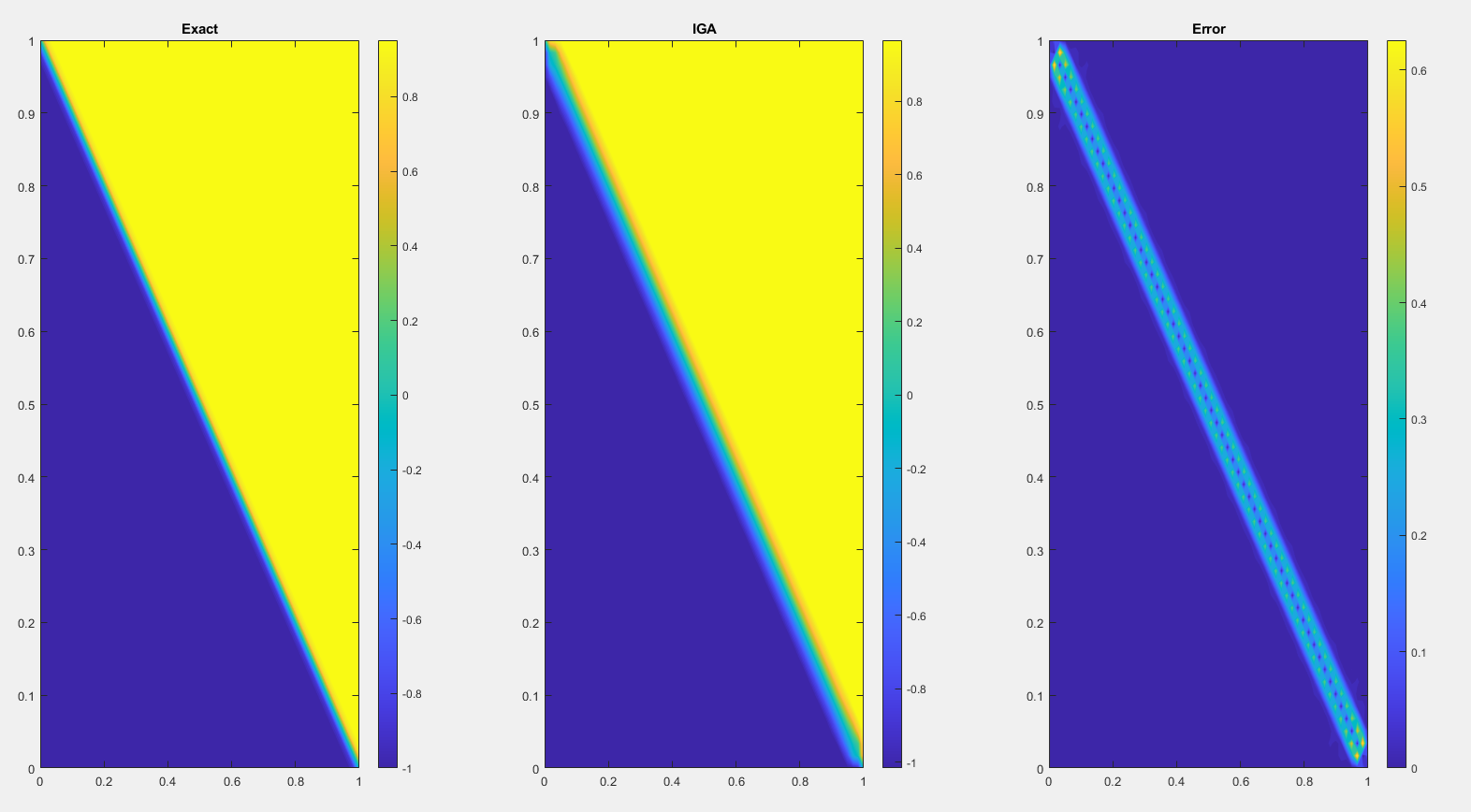}
\caption{IGA-ODIL solution for the Allen--Cahn equation.}
\end{figure}

\begin{table}[ht]
\centering
\caption{Results for the Allen--Cahn problem solved with sparse IGA--ODIL/Newton iterations.}
\label{tab:allencahn_iga_odil}
\begin{tabular}{ccccc}
\hline
$p$ & $n_x=n_y$ & $N_c$ & Time [s] & $L^2$ error \\
\hline
1 & 10 & 60 & $5.139\times 10^{-2}$ & $2.155\times 10^{-1}$ \\
1 & 20 & 60 & $8.132\times 10^{-2}$ & $1.085\times 10^{-1}$\\
1 & 30 & 60 & $1.093\times 10^{-1}$ & $6.421\times 10^{-2}$ \\
1 & 40 & 60 & $1.799\times 10^{-1}$ & $4.750\times 10^{-2}$ \\
2 & 10 & 60 & $4.190\times 10^{-2}$ & $2.076\times 10^{-1}$ \\
2 & 20 & 60 & $9.090\times 10^{-2}$ & $1.069\times 10^{-1}$ \\
2 & 30 & 60 & $1.931\times 10^{-1}$ & $7.879\times 10^{-2}$ \\
2 & 40 & 60 & $3.123\times 10^{-1}$ & $8.191\times 10^{-2}$ \\
3 & 10 & 60 & $6.398\times 10^{-2}$ & $2.071\times 10^{-1}$ \\
3 & 20 & 60 & $1.330\times 10^{-1}$ & $1.046\times 10^{-1}$\\
3 & 30 & 60 & $2.403\times 10^{-1}$ & $7.824\times 10^{-2}$ \\
3 & 40 & 60 & $3.375\times 10^{-1}$ & $8.594\times 10^{-2}$ \\
\hline
\end{tabular}
\end{table}

\begin{figure}
\includegraphics[width=0.49\textwidth]{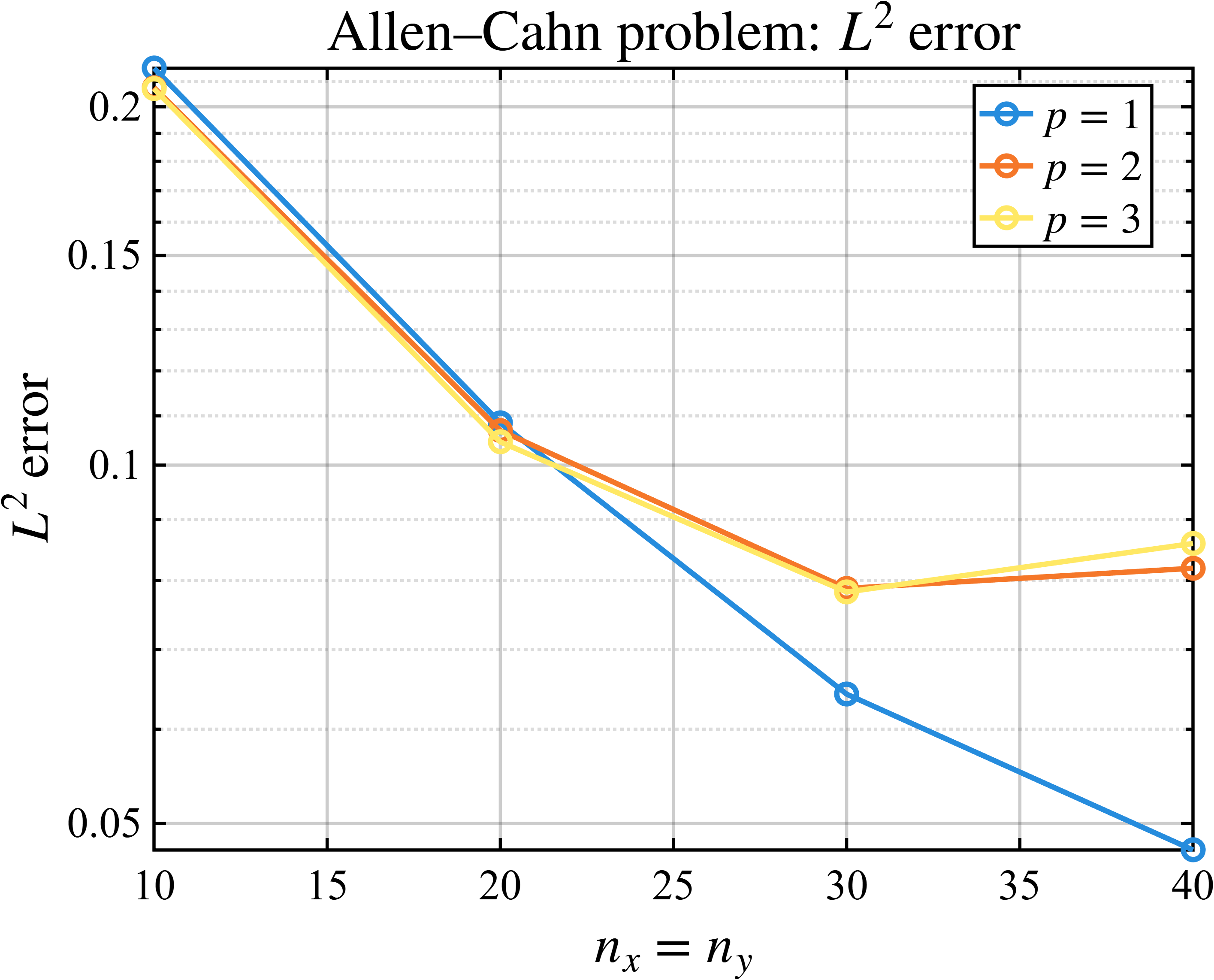}\includegraphics[width=0.49\textwidth]{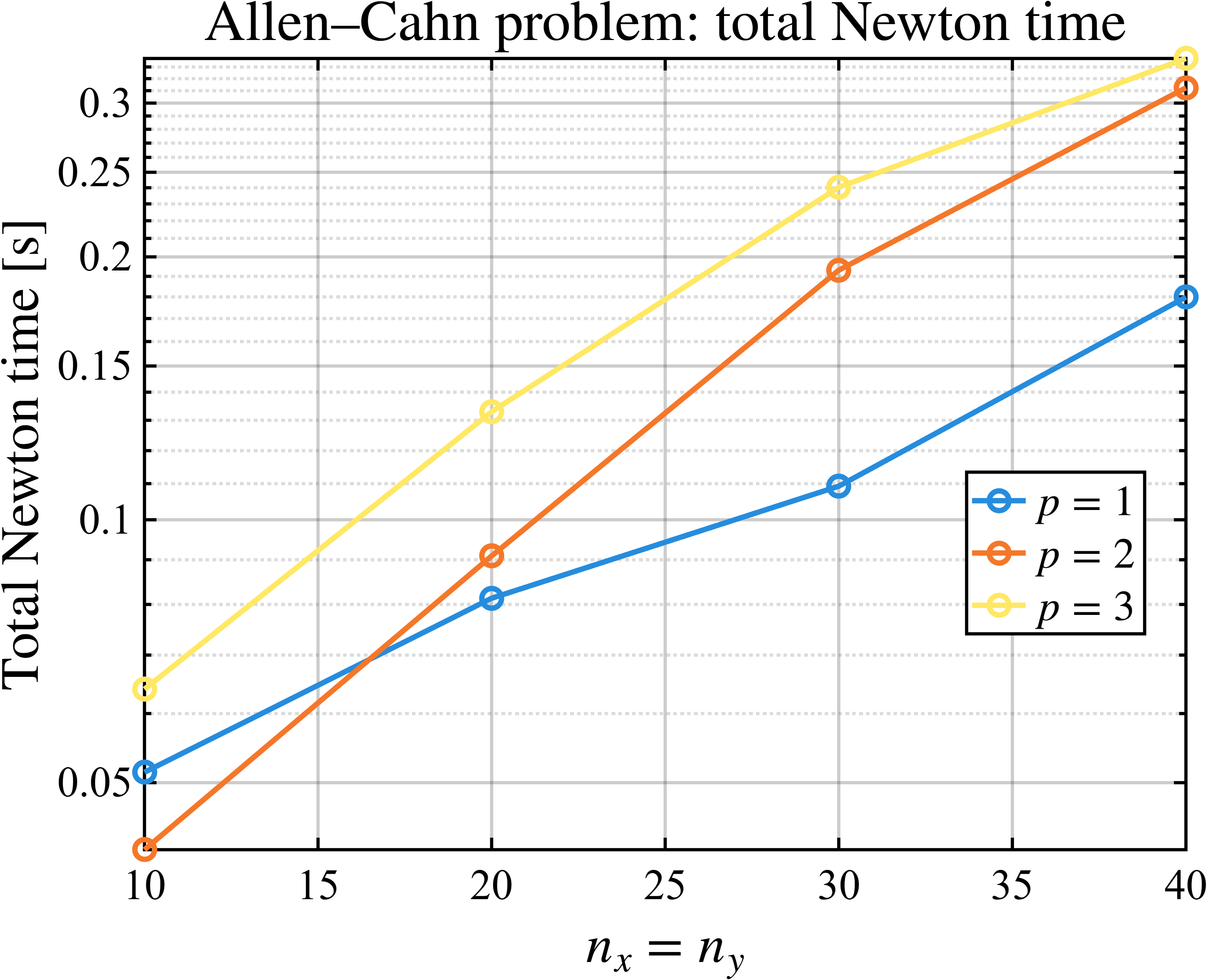}
\caption{Accuracy and timing of IGA-ODIL solution for the Allen-Cahn problem for $\epsilon=0.01$ for B-splines of order $p=1,2,3$.}
\label{fig:num3}
\end{figure}

\subsection{Eriksson--Johnson Problem on a Circular Domain}

We consider the convection--diffusion equation posed on the unit disk
$
\Omega = \{ (x,y)\in\mathbb{R}^2 : x^2+y^2 \le 1 \},
$
given by
\begin{equation}
-\varepsilon \Delta u + \boldsymbol{\beta}\cdot\nabla u = 0
\quad \text{in } \Omega,
\end{equation}
where $\varepsilon > 0$ denotes the diffusion coefficient and
$
\boldsymbol{\beta} = (1,0)^T
$
is the convection direction.
The computational domain is parameterized by a tensor-product spline space on the reference square
$
(\xi,\eta)\in [0,1]^2.
$
To obtain a circular physical domain, we employ a smooth square-to-disk transformation. We first define
\begin{align}
a = 2\xi - 1, \;
b = 2\eta - 1,
\end{align}
and then we introduce the mapping
\begin{align}
x = a \sqrt{1-\frac{b^2}{2}}, \;
y = b \sqrt{1-\frac{a^2}{2}}.
\end{align}

This mapping transforms the unit square into the unit disk while avoiding the singularity associated with polar coordinates at the origin. In particular the mapping is smooth, and no coordinate singularity appears at the center. Additionally, tensor-product spline discretizations can still be used, and the collocation points remain structured in parameter space.
The PDE is solved in the parametric domain, while the numerical solution is visualized on the physical circular geometry.
The numerical solution is represented as
\begin{equation}
u_h(\xi,\eta)
=
\sum_{i=1}^{n_x}
\sum_{j=1}^{n_y}
c_{j,i}
B_i(\xi) B_j(\eta),
\end{equation}
where $B_i$ and $B_j$ are tensor-product B-spline basis functions.
In the present implementation, the strong-form residual is evaluated in parameter space:
\begin{equation}
R(u_h)
=
-\varepsilon
\left(
\frac{\partial^2 u_h}{\partial \xi^2}
+
\frac{\partial^2 u_h}{\partial \eta^2}
\right)
+
\frac{\partial u_h}{\partial \xi}.
\end{equation}

The spline coefficients are obtained by minimizing the least-squares functional
\begin{equation}
\mathcal{J}(c)
=
\frac12
\sum_{k=1}^{N_c}
|R(u_h)(\xi_k,\eta_k)|^2,
\end{equation}
evaluated at collocation points
$
\{(\xi_k,\eta_k)\}_{k=1}^{N_c}.
$
Dirichlet boundary conditions are imposed strongly through elimination of boundary spline coefficients. The inflow boundary condition is prescribed on the left side of the parametric domain:
\begin{equation}
u(0,\eta)=\sin(\pi\eta),
\end{equation}
while homogeneous conditions are imposed on the remaining boundaries.
We use $40\times 40$ elements and cubic B-splines with $200\times 200$ collocation points for $\epsilon=0.1$. The problem takes 0.15 sec on single core MATLAB with Gauss-Newton method. 
The solution is presented in Figure \ref{fig:circ}.
This formulation demonstrates how residual minimization in spline spaces can naturally incorporate nontrivial geometric mappings while preserving the tensor-product structure of the discretization.
\begin{figure}
\includegraphics[width=\textwidth]{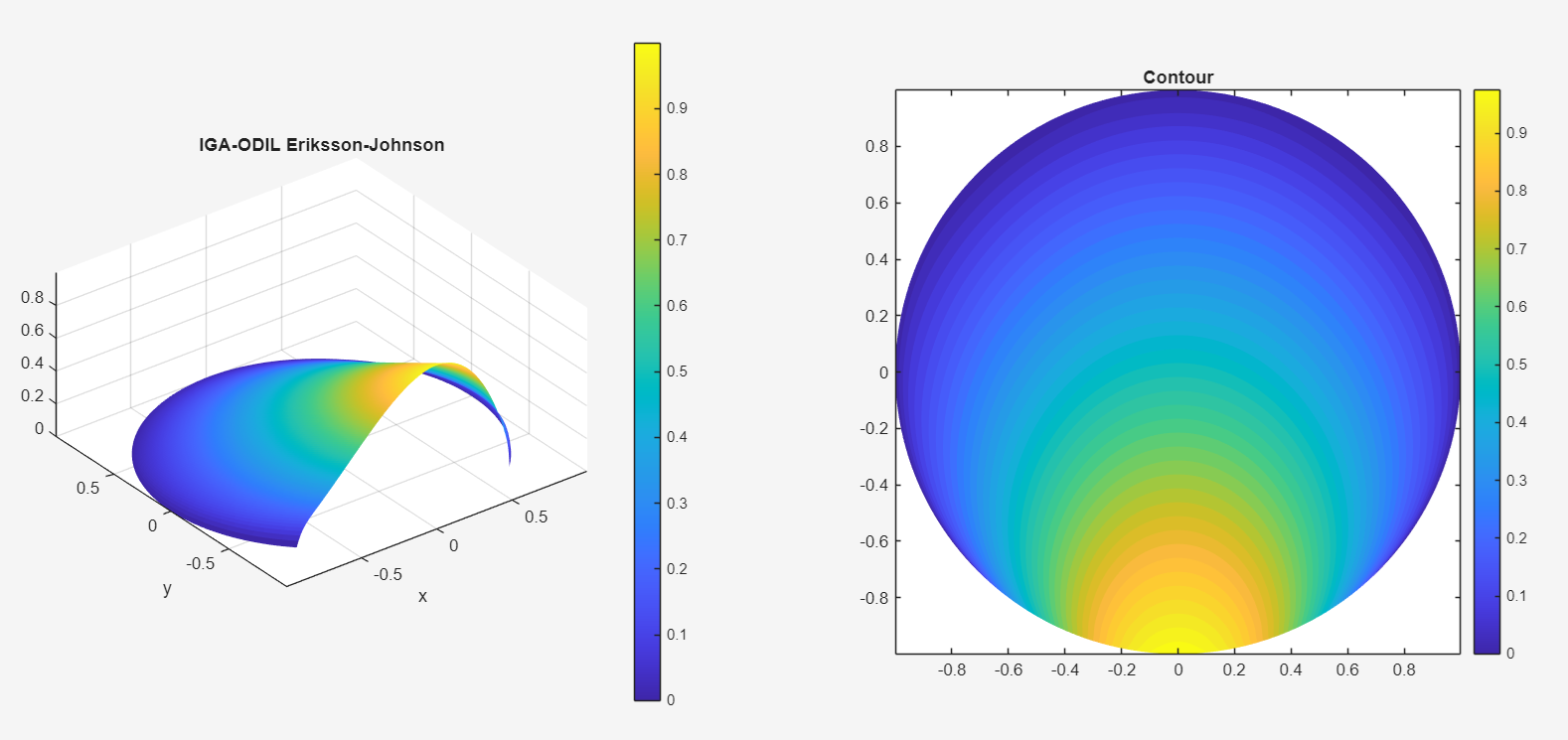}
\caption{Eriksson-Johnson problem solved in a circular domain}
\label{fig:circ}
\end{figure}

\subsection{Helmholtz Equation on a Ball}

We next consider the three-dimensional Helmholtz equation posed on the unit ball
$
\Omega = \{ (x,y,z)\in\mathbb{R}^3 : x^2+y^2+z^2 \le 1 \}.
$
The governing equation reads
\begin{equation}
-\Delta u + \kappa^2 u = f
\qquad \text{in } \Omega,
\end{equation}
subject to homogeneous Dirichlet boundary conditions
\begin{equation}
u = 0
\qquad \text{on } \partial\Omega.
\end{equation}

The geometry of the ball is represented through a smooth mapping from the parametric cube
$
(\xi,\eta,\zeta)\in [0,1]^3
$
to the physical domain. Introducing
$
a = 2\xi-1,
\;
b = 2\eta-1,
\;
c = 2\zeta-1,
$ the cube-to-ball transformation is defined by
\begin{align}
x &= a\sqrt{
1-\frac{b^2}{2}-\frac{c^2}{2}
+\frac{b^2c^2}{3}}, \\
y &= b\sqrt{
1-\frac{a^2}{2}-\frac{c^2}{2}
+\frac{a^2c^2}{3}}, \\
z &= c\sqrt{
1-\frac{a^2}{2}-\frac{b^2}{2}
+\frac{a^2b^2}{3}}.
\end{align}

This mapping transforms the unit cube into a smooth approximation of the unit ball while preserving tensor-product spline parameterization.
The approximate solution is represented in the parametric domain using tensor-product B-splines:
\begin{equation}
u_h(\xi,\eta,\zeta)
=
\sum_{i=1}^{n_x}
\sum_{j=1}^{n_y}
\sum_{k=1}^{n_z}
C_{ijk}
B_i(\xi)
B_j(\eta)
B_k(\zeta),
\end{equation}
where $B_i$, $B_j$, and $B_k$ denote univariate B-spline basis functions.
For verification purposes, we employ the manufactured solution
\begin{equation}
u(\xi,\eta,\zeta)
=
\sin(\kappa\pi\xi)
\sin(\kappa\pi\eta)
\sin(\kappa\pi\zeta),
\end{equation}
which yields the forcing term
\begin{equation}
f(\xi,\eta,\zeta)
=
\left(
3\kappa^2\pi^2+\kappa^2
\right)
u(\xi,\eta,\zeta).
\end{equation}

The residual minimization formulation is constructed directly in the parametric domain:
\begin{equation}
R(\xi,\eta,\zeta)
=
-\Delta_{\xi,\eta,\zeta} u_h
+
\kappa^2 u_h
-
f,
\end{equation}
and the spline coefficients are obtained by minimizing the discrete least-squares functional
\begin{equation}
\mathcal{J}(C)
=
\frac12
\sum_{m=1}^{N_c}
R(\xi_m,\eta_m,\zeta_m)^2,
\end{equation}
evaluated at collocation points inside the parametric cube.
The resulting linear system is solved after elimination of boundary degrees of freedom corresponding to the cube boundary, which maps onto the spherical boundary in physical space.
First we use $16\times 16\times 16$ elements and cubic B-splines with $30\times 30\times 30$ collocation points for $\kappa=6$. The problem takes 0.74 sec on single core MATLAB with Gauss-Newton method. The accuracy in $L^2$ norm is 5.425967e-03. 
Figure~\ref{fig:helmholtz_ball} presents multi-slice visualizations of the numerical solution inside the ball geometry. The smooth spline representation accurately resolves oscillatory Helmholtz modes while preserving geometric regularity of the mapped domain.

Second. we use $32\times 32\times 32$ elements with cubic B-splines with $64\times 64 \times 64$ collocation points for $\kappa=20$. The problem takes 23.43 sec on single core MATLAB with Gauss-Newton method. The resulting accuracy in $L^2$ norm is 4.742797e-02. This problem solution is illustrated in Figure \ref{fig:ball20}.

\begin{figure}[h!]
\centering
\includegraphics[width=0.9\textwidth]{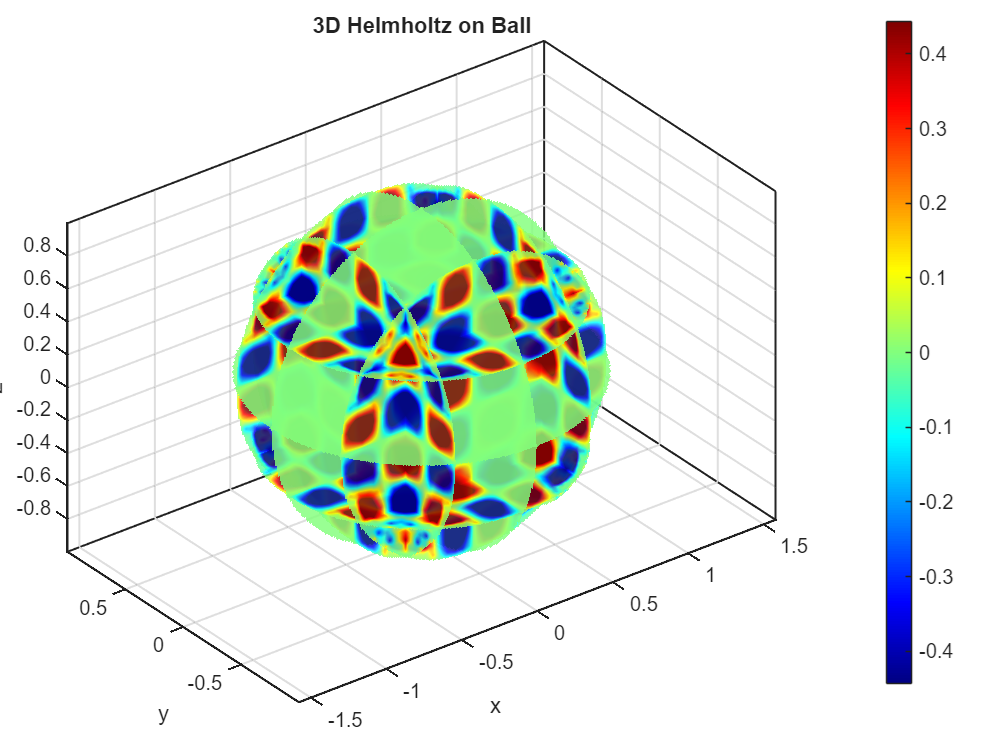}
\caption{Multi-slice visualization of the 3D Helmholtz solution inside the mapped ball geometry using cubic B-splines and residual minimization.  $16\times 16\times 16$ elements and cubic B-splines with $30\times 30\times 30$ collocation points for $\kappa=6$. } 
\label{fig:helmholtz_ball}
\end{figure}

\begin{figure}[h!]
\centering
\includegraphics[width=0.9\textwidth]{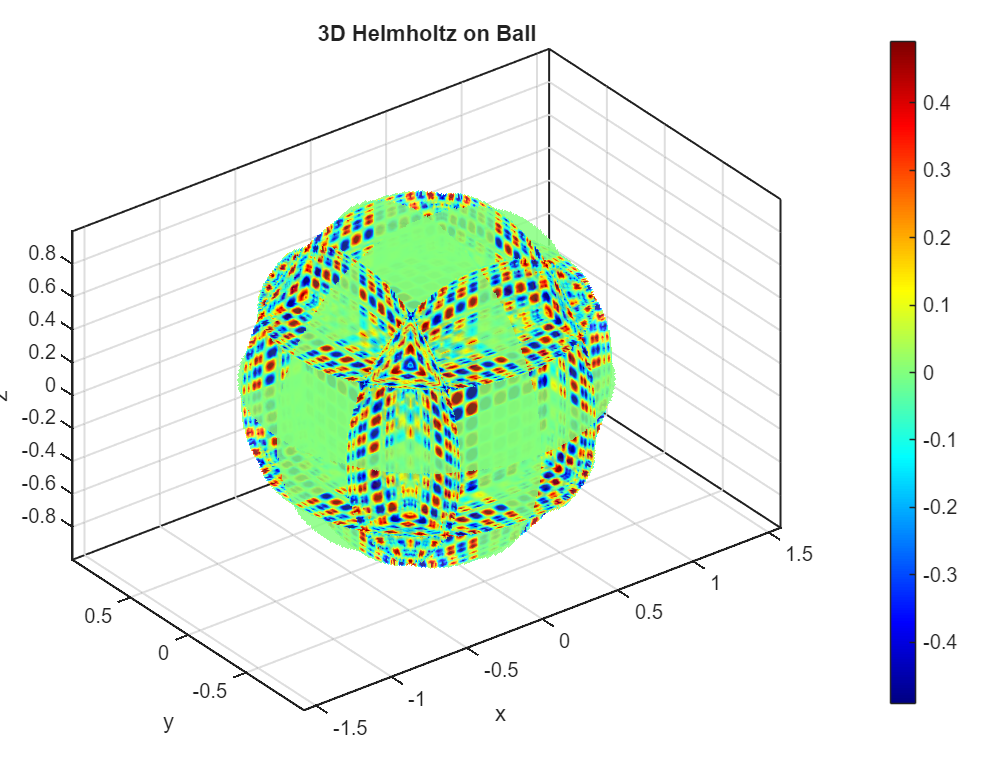}
\caption{Multi-slice visualization of the 3D Helmholtz solution inside the mapped ball geometry using cubic B-splines and residual minimization.  $32\times 32\times 32$ elements and cubic B-splines with $64\times 64\times 64$ collocation points for $\kappa=20$. } 
\label{fig:ball20}
\end{figure}
\section{Inverse Problems}

IGA-ODIL naturally extends to PDE-constrained inverse problems due to its residual-based optimization formulation. In contrast to classical finite element approaches, where inverse problems often require repeated nested forward PDE solves, the proposed framework simultaneously optimizes both the state variables and unknown physical parameters within a unified optimization procedure.
We consider inverse identification of parameters in the Helmholtz equation:
\begin{equation}
-\Delta u + \alpha u = f,
\quad (x,y)\in\Omega,
\label{eq:helmholtz_inverse}
\end{equation}
subject to homogeneous Dirichlet boundary conditions
\begin{equation}
u(x,y)=0,
\quad (x,y)\in\partial\Omega.
\end{equation}

A manufactured solution is prescribed as
\begin{equation}
u(x,y)=
\sin(\kappa_{\mathrm{true}}\pi x)
\sin(\kappa_{\mathrm{true}}\pi y),
\end{equation}
where $\kappa_{\mathrm{true}}$ denotes the unknown frequency parameter to be recovered.
The corresponding forcing term becomes
\begin{equation}
f(x,y)=
\left(
2(\kappa_{\mathrm{true}}\pi)^2+\alpha
\right)
\sin(\kappa_{\mathrm{true}}\pi x)
\sin(\kappa_{\mathrm{true}}\pi y).
\end{equation}

The objective of the inverse problem is to recover the unknown parameter $\kappa$ from observations of the solution field.

\subsection{IGA-ODIL Formulation for Inverse Problems}

The spline approximation is written as
\begin{equation}
u_c(x,y)=
\sum_{i=1}^{N_x}
\sum_{j=1}^{N_y}
c_{ij}
B_i^{(p)}(x)
B_j^{(q)}(y).
\label{eq:inverse_spline}
\end{equation}
The inverse optimization problem is formulated as
\begin{equation}
\min_{c,\kappa}
\left(
\|R(c,\kappa)\|^2
+
\lambda
\|u_c-u_{\mathrm{obs}}\|^2
\right),
\label{eq:inverse_objective}
\end{equation}
where $\lambda>0$ is a regularization parameter and $u_{\mathrm{obs}}$ denotes observed solution data.
The spline coefficients $c_{ij}$ and parameter $\kappa$ are optimized simultaneously.
The residual is given by
\begin{equation}
R_k(c,\kappa)=
-\Delta u_c(x_k,y_k)
+
\kappa^2 u_c(x_k,y_k)
-
f(x_k,y_k).
\end{equation}

Expanding the spline representation yields
\begin{align}
R_k(c,\kappa)
=
&
\sum_{i,j}
c_{ij}
\Bigg[
-
\frac{\partial^2 B_i^{(p)}(x_k)}{\partial x^2}
B_j^{(q)}(y_k)
\nonumber\\
&
-
B_i^{(p)}(x_k)
\frac{\partial^2 B_j^{(q)}(y_k)}{\partial y^2}
+
\kappa^2
B_i^{(p)}(x_k)
B_j^{(q)}(y_k)
\Bigg]
-f(x_k,y_k).
\end{align}

The data misfit term becomes
\begin{equation}
\|u_c-u_{\mathrm{obs}}\|^2
=
\sum_k
\left(
u_c(x_k,y_k)-u_{\mathrm{obs}}(x_k,y_k)
\right)^2.
\end{equation}

The complete optimization functional is therefore
\begin{align}
\Phi(c,\kappa)
=
\frac12
\sum_k
R_k(c,\kappa)^2
+
\frac{\lambda}{2}
\sum_k
\left(
u_c(x_k,y_k)-u_{\mathrm{obs}}(x_k,y_k)
\right)^2.
\label{eq:inverse_full_loss}
\end{align}

\subsection{Gauss--Newton Linearization}

Let
\begin{equation}
\theta=
\begin{bmatrix}
c \\
\kappa
\end{bmatrix}
\end{equation}
denote the vector of optimization variables.
The residual vector is linearized as
\begin{equation}
R(\theta+\delta\theta)
\approx
R(\theta)
+
J(\theta)\delta\theta,
\end{equation}
where
\begin{equation}
J(\theta)=
\frac{\partial R}{\partial\theta}
\end{equation}
is the Jacobian matrix.
The Gauss-Newton system becomes
\begin{equation}
J^T J\,\delta\theta
=
-J^T R.
\label{eq:inverse_gn}
\end{equation}

Due to the local support of B-spline basis functions, the Jacobian remains sparse with respect to spline coefficients. The additional parameter $\kappa$ introduces only one dense column corresponding to
\begin{equation}
\frac{\partial R_k}{\partial\kappa}
=
2\kappa
u_c(x_k,y_k).
\end{equation}

Consequently, the resulting system preserves favorable sparsity properties and remains computationally tractable.

\subsection{Numerical Results}

We first consider the recovery of the parameter
$\kappa_{\mathrm{true}}=2.0$
starting from the initial guess
$\kappa_{\mathrm{initial}}=1.0$.
The computation employs $30\times30$ elements,
cubic B-splines,
$40\times40$ collocation points, and 100 Gauss–Newton iterations.
Figure~\ref{fig:inverse_conv_kappa2} illustrates the convergence of the recovered parameter.
The corresponding reconstructed solution is shown in Figure~\ref{fig:inverse_sol_kappa2}.
\begin{figure}[ht]
\centering
\includegraphics[width=0.9\textwidth]{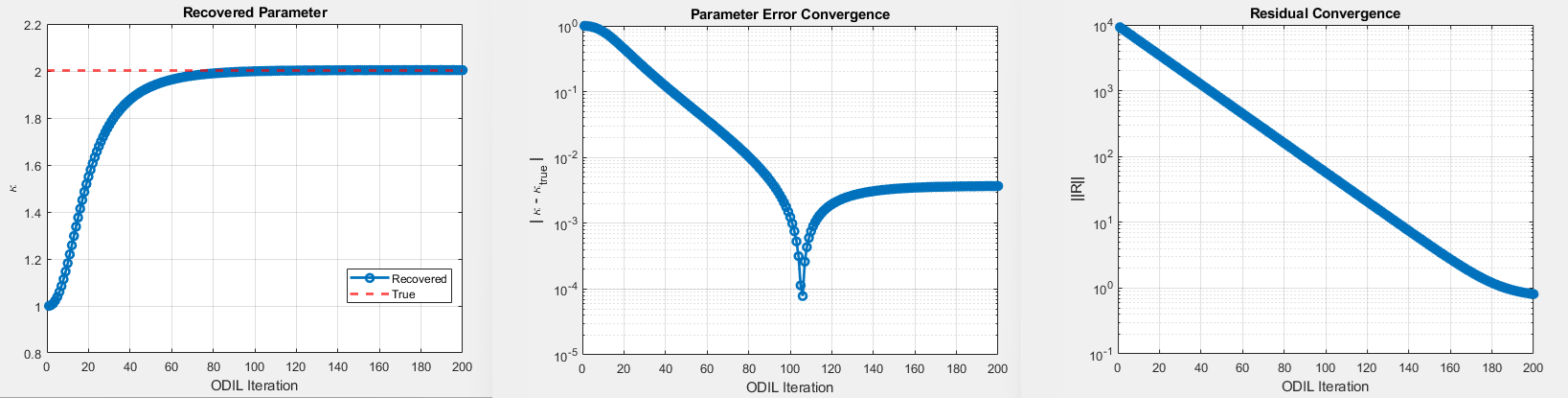}
\caption{Convergence of inverse Helmholtz parameter recovery for $\kappa_{\mathrm{true}}=2$.}
\label{fig:inverse_conv_kappa2}
\end{figure}
\begin{figure}[ht]
\centering
\includegraphics[width=0.95\textwidth]{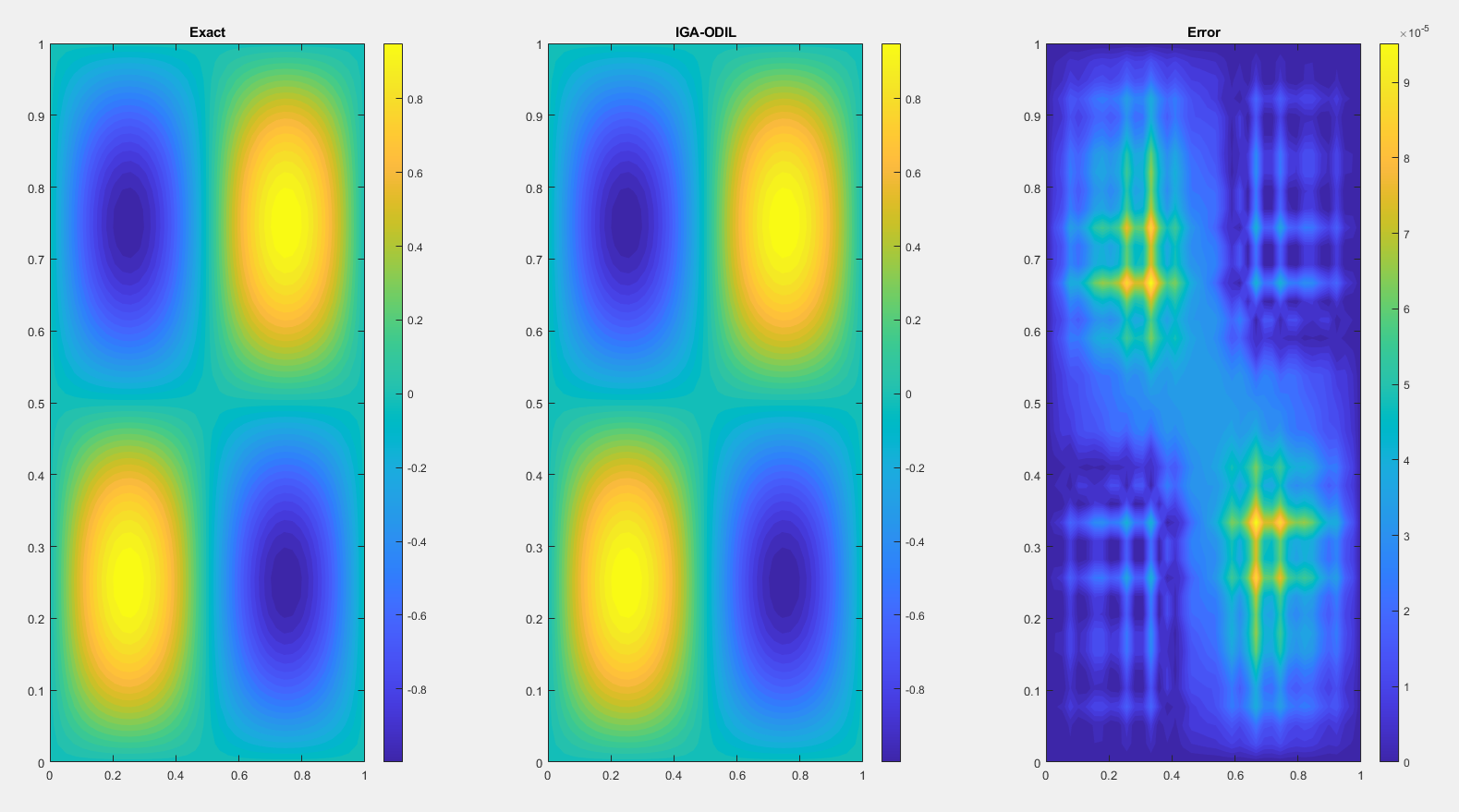}
\caption{Recovered Helmholtz solution for $\kappa_{\mathrm{true}}=2$.}
\label{fig:inverse_sol_kappa2}
\end{figure}

The parameter converges rapidly toward the exact value while the residual decreases monotonically.
We next consider a substantially more oscillatory regime with
$\kappa_{\mathrm{true}}=8$.
The computational setup uses  $40\times40$ spline elements,
cubic B-splines,
$80\times80$ collocation points, and
 100 Gauss-Newton iterations.
Figure~\ref{fig:inverse_conv_kappa8} presents the convergence histories for the recovered parameter.
The reconstructed oscillatory solution is shown in Figure~\ref{fig:inverse_sol_kappa8}.

\begin{figure}[ht]
\centering
\includegraphics[width=0.9\textwidth]{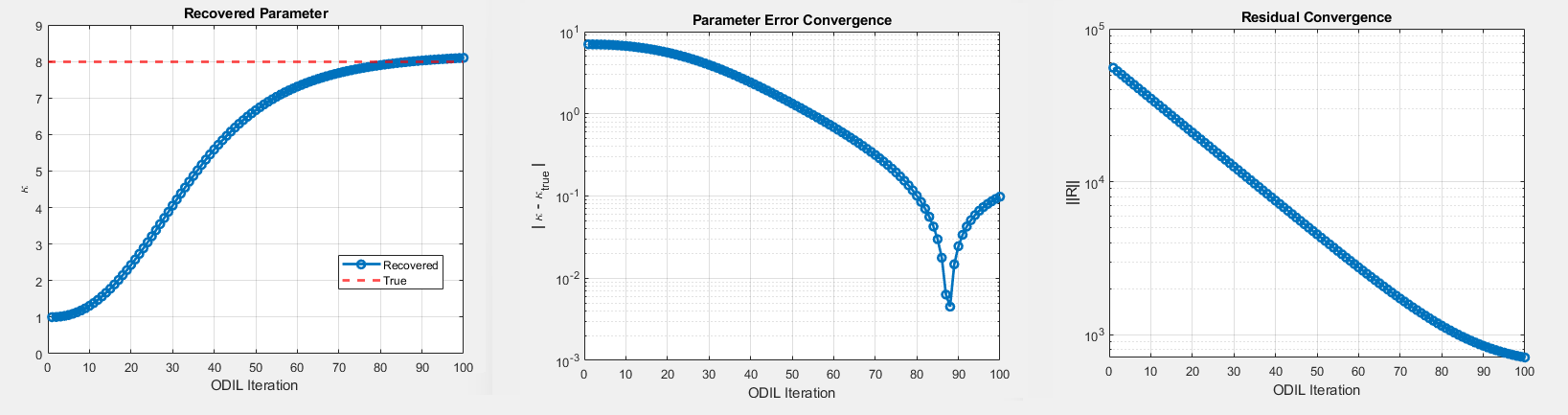}
\caption{Convergence of inverse Helmholtz parameter recovery for $\kappa_{\mathrm{true}}=8$.}
\label{fig:inverse_conv_kappa8}
\end{figure}
\begin{figure}[ht]
\centering
\includegraphics[width=0.95\textwidth]{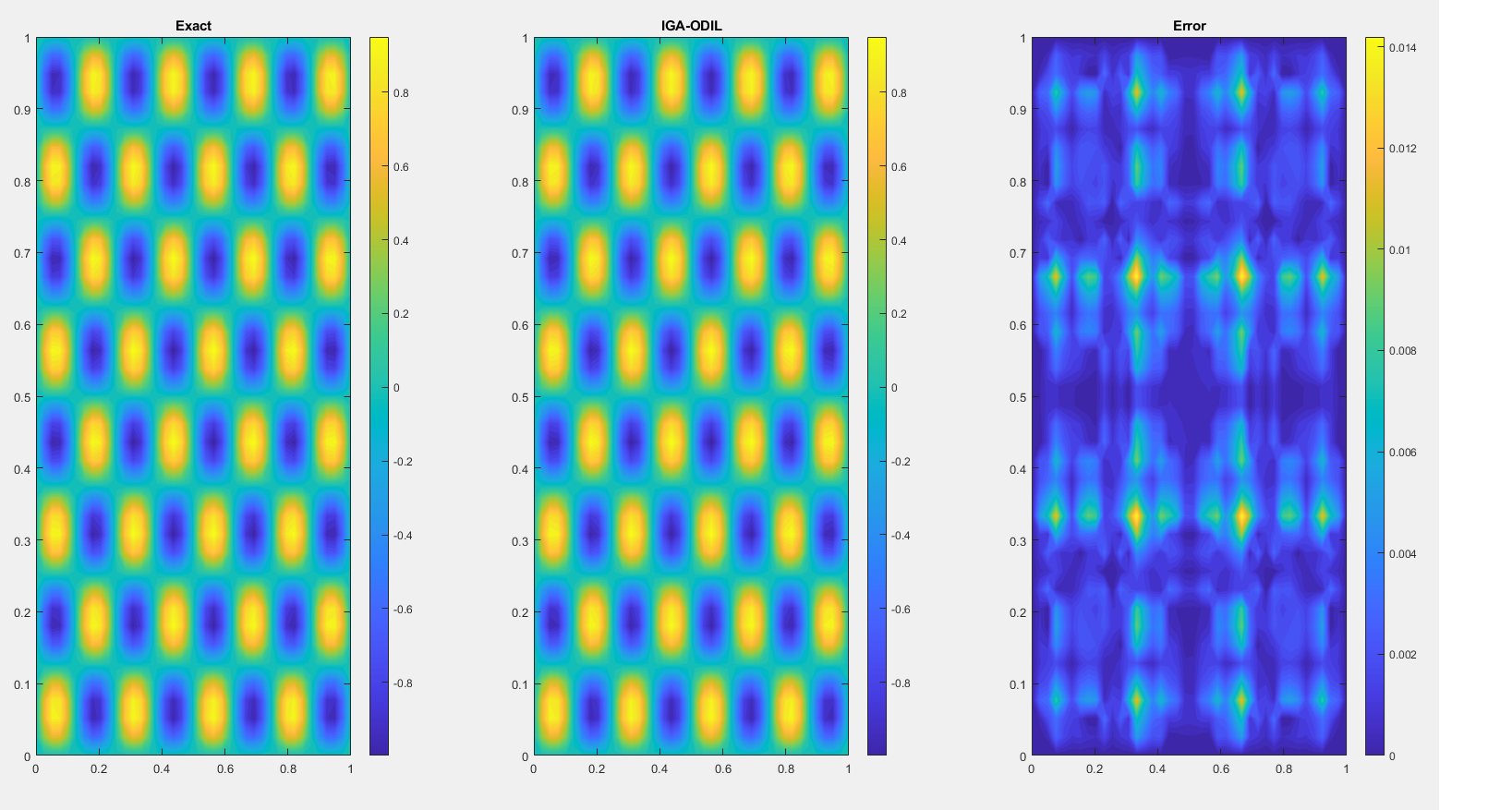}
\caption{Recovered Helmholtz solution for $\kappa_{\mathrm{true}}=8$.}
\label{fig:inverse_sol_kappa8}
\end{figure}
Even in highly oscillatory regimes, the proposed formulation successfully recovers both the parameter and solution field.
The computational time is approximately $9.8$ seconds for $\kappa=2$, and  $59$ seconds for $\kappa=8$,
on a single CPU core in MATLAB.
These results demonstrate that IGA-ODIL naturally extends to PDE-constrained inverse problems while preserving the favorable sparsity and optimization properties of the forward formulation.

\section{Conclusions}

The proposed IGA-ODIL framework bridges scientific machine learning and classical numerical PDE discretization theory.
The key conceptual observation is that residual minimization becomes computationally tractable when combined with structured spline parameterizations.
The numerical experiments demonstrate that the proposed method provides a simple and flexible framework for solving a wide range of PDEs. By operating directly on the strong form of the governing equations, the method avoids the need for variational formulations and numerical quadrature, which simplifies implementation compared to classical Galerkin-based approaches.
The use of spline basis functions yields smooth approximations with high-order continuity, enabling accurate evaluation of differential operators and reducing the number of degrees of freedom required for a given accuracy. This is particularly advantageous for problems involving higher-order derivatives or smooth solutions.
Furthermore, the formulation naturally leads to a least-squares problem with a structured Jacobian, allowing the use of efficient Gauss–Newton or Newton-type solvers. For linear problems, the method reduces to a single linear solve, while for nonlinear problems, it exhibits rapid (often quadratic) convergence under standard assumptions.

The proposed framework is closely related to least-squares finite element methods (LSFEM) \cite{bochev2009least}, but differs in the choice of approximation space. While LSFEM typically employs low-order polynomial basis functions, the present approach utilizes smooth spline spaces from isogeometric analysis, leading to improved approximation properties.
Compared to classical isogeometric Galerkin methods \cite{iga}, the present formulation eliminates the need for weak forms and numerical integration, providing a more direct residual-based approach. However, this also implies that the method inherits some characteristics of collocation-type methods.
In contrast to physics-informed neural networks (PINNs) \cite{raissi}, the method uses a deterministic basis representation with a structured Jacobian, resulting in improved conditioning, reproducibility, and convergence behavior.
The IGA-ODIL differs from the ODIL method \cite{odil} since it uses spline combination instead of discrete point values for training (optimization).

Despite its advantages, the proposed method has several limitations that warrant further investigation.
First, the use of strong-form residuals requires sufficient regularity of the exact solution, as the differential operator is applied directly to the approximate solution. This may limit applicability to problems with low regularity or discontinuities.
Second, the collocation-based approximation of the residual introduces a dependence on the choice and distribution of collocation points. While uniform grids perform well for smooth problems, more complex geometries or localized features may require adaptive sampling strategies.
Third, for convection-dominated problems with very small diffusion coefficients, additional stabilization mechanisms or weighted residual formulations may be necessary to fully resolve sharp boundary layers without oscillations.
Fourth, the formation of the Jacobian matrix can be computationally expensive for large-scale problems, particularly in higher dimensions. Although the structure of the spline basis can be exploited, further work is needed to develop efficient matrix-free or low-rank implementations.
Finally, while Gauss-Newton iterations perform well in the examples considered, global convergence is not guaranteed for strongly nonlinear problems. In such cases, damping strategies or line search techniques may be required.

\appendix
\section{Code availability}

The MATLAB codes employed in the paper to solve IGA-ODIL problems are available at

{\tt https://github.com/sluzalec/IGA-ODIL}

The Python codes used for running PINN and CRVPINN experiments are available in paper \cite{crvpinn}.

\section{Convergence of IGA-ODIL}

\begin{theorem}[Convergence of robust spline-based residual minimization]
\label{thm:robust_iga_odil}

Let $\Omega \subset \mathbb{R}^d$ be a bounded Lipschitz domain, and consider the boundary value problem
\begin{equation}
\mathcal{L}u = f
\quad \text{in } \Omega,
\label{eq:general_pde}
\end{equation}
supplemented with boundary conditions such that the problem is well posed.

Assume that:

\begin{enumerate}

\item
$\mathcal{L}:H^m(\Omega)\to L^2(\Omega)$
is a linear differential operator of order
$
s\le m.
$

\item
The exact solution satisfies
$
u\in H^{p+1}(\Omega).
$

\item
The spline approximation space
$
V_h\subset H^m(\Omega)
$
consists of tensor-product B-splines of degree
$
p
$
constructed on a quasi-uniform mesh with characteristic size
$
h.
$

\item
The collocation points
$
\mathcal{X}_M=\{x_k\}_{k=1}^M\subset\Omega
$
are quasi-uniform with fill distance
$
h_c
:=
\sup_{x\in\Omega}
\min_{x_k\in\mathcal{X}_M}
\|x-x_k\|,
$
satisfying
$
h_c\lesssim h.
$

\item
The discrete robust residual minimization problem
\begin{equation}
u_h
=
\arg\min_{v_h\in V_h}
\Phi_h(v_h)
\label{eq:robust_discrete_problem}
\end{equation}
is defined by the functional
\begin{equation}
\Phi_h(v_h)
=
\frac12
R(v_h)^T
G^{-1}
R(v_h),
\label{eq:robust_functional}
\end{equation}
where
\begin{equation}
R(v_h)
=
\left[
\mathcal{L}v_h(x_1)-f(x_1),
\dots,
\mathcal{L}v_h(x_M)-f(x_M)
\right]^T
\in\mathbb{R}^M,
\end{equation}
and
$
G\in\mathbb{R}^{M\times M}
$
is a symmetric positive definite Gram matrix.

\item
The operator $\mathcal{L}$ satisfies the stability estimate
\begin{equation}
\|w\|_{H^m(\Omega)}
\le
C_{\mathrm{stab}}
\|\mathcal{L}w\|_{L^2(\Omega)}
\quad
\forall w\in H_0^m(\Omega).
\label{eq:stability_assumption_theorem}
\end{equation}

\item
The collocation sampling satisfies the norm equivalence
\begin{equation}
\|v_h\|_{L^2(\Omega)}^2
\sim
\sum_{k=1}^{M}
|v_h(x_k)|^2
\quad
\forall v_h\in V_h,
\label{eq:sampling_equivalence_theorem}
\end{equation}
with constants independent of $h$.

\item
The Gram matrix induces a uniformly equivalent discrete norm:
\begin{equation}
c_G\|r\|_2^2
\le
r^TG^{-1}r
\le
C_G\|r\|_2^2
\quad
\forall r\in\mathbb{R}^M,
\label{eq:gram_equivalence_theorem}
\end{equation}
with constants independent of $h$.

\end{enumerate}

Then the robust IGA--ODIL approximation satisfies the error estimate
\begin{equation}
\|u-u_h\|_{H^m(\Omega)}
\le
C
\|R(u_h)\|_G,
\label{eq:error_equivalence}
\end{equation}
where
\begin{equation}
\|R(u_h)\|_G
:=
\left(
R(u_h)^TG^{-1}R(u_h)
\right)^{1/2}.
\label{eq:robust_norm}
\end{equation}

Moreover,
\begin{equation}
\|R(u_h)\|_G
\le
C
h^{p+1-s}
|u|_{H^{p+1}(\Omega)},
\label{eq:residual_convergence_theorem}
\end{equation}
and therefore
\begin{equation}
\|u-u_h\|_{H^m(\Omega)}
=
\mathcal{O}(h^{p+1-s}).
\label{eq:final_rate_theorem}
\end{equation}

\end{theorem}

\begin{proof}

\noindent
\textbf{Step 1: Spline approximation property.}

Since
$
u\in H^{p+1}(\Omega),
$
classical spline approximation theory
(cf.\ Ciarlet~\cite{ciarlet2002finite},
Schumaker~\cite{schumaker2007spline},
Bazilevs et al.~\cite{bazilevs2006isogeometric})
implies the existence of a spline interpolant
$
\Pi_hu\in V_h
$
such that
\begin{equation}
\|u-\Pi_hu\|_{H^m(\Omega)}
\le
Ch^{p+1-m}
|u|_{H^{p+1}(\Omega)}.
\label{eq:proof_interp}
\end{equation}

Since $\mathcal{L}$ is a differential operator of order $s$,
standard continuity estimates yield
\begin{equation}
\|\mathcal{L}(u-\Pi_hu)\|_{L^2(\Omega)}
\le
Ch^{p+1-s}
|u|_{H^{p+1}(\Omega)}.
\label{eq:proof_operator_error}
\end{equation}

\noindent
\textbf{Step 2: Minimality of the robust residual functional.}

Since
$
u_h
=
\arg\min_{v_h\in V_h}
\Phi_h(v_h),
$
we have
\begin{equation}
\Phi_h(u_h)
\le
\Phi_h(\Pi_hu).
\label{eq:proof_minimality}
\end{equation}

Using the definition of the robust functional,
$
\Phi_h(v_h)
=
\frac12
R(v_h)^TG^{-1}R(v_h),
$
this becomes
\begin{equation}
R(u_h)^TG^{-1}R(u_h)
\le
R(\Pi_hu)^TG^{-1}R(\Pi_hu).
\label{eq:proof_robust_minimality}
\end{equation}

Since the exact solution satisfies
$
\mathcal{L}u=f,
$
the residual associated with the interpolant satisfies
$
R(\Pi_hu)
=
\left[
\mathcal{L}(\Pi_hu-u)(x_k)
\right]_{k=1}^{M}.
$

Using the Gram norm equivalence
\eqref{eq:gram_equivalence_theorem},
\begin{equation}
R(\Pi_hu)^TG^{-1}R(\Pi_hu)
\le
C_G
\sum_{k=1}^{M}
|\mathcal{L}(\Pi_hu-u)(x_k)|^2.
\label{eq:proof_gram_upper}
\end{equation}

Applying the sampling equivalence
\eqref{eq:sampling_equivalence_theorem},
\begin{equation}
\sum_{k=1}^{M}
|\mathcal{L}(\Pi_hu-u)(x_k)|^2
\le
C
\|\mathcal{L}(\Pi_hu-u)\|_{L^2(\Omega)}^2.
\label{eq:proof_sampling}
\end{equation}

Combining
\eqref{eq:proof_operator_error},
\eqref{eq:proof_gram_upper},
and
\eqref{eq:proof_sampling}
yields
\begin{equation}
R(u_h)^TG^{-1}R(u_h)
\le
Ch^{2(p+1-s)}
|u|_{H^{p+1}(\Omega)}^2.
\label{eq:proof_residual_bound_sq}
\end{equation}

Taking square roots gives
\begin{equation}
\|R(u_h)\|_G
\le
Ch^{p+1-s}
|u|_{H^{p+1}(\Omega)}.
\label{eq:proof_residual_bound}
\end{equation}

This proves
\eqref{eq:residual_convergence_theorem}.

\noindent
\textbf{Step 3: Stability estimate.}

Define the approximation error
$
e_h:=u-u_h.
$
Using linearity of the operator,
$
\mathcal{L}e_h
=
f-\mathcal{L}u_h.
$
Applying the stability estimate
\eqref{eq:stability_assumption_theorem},
\begin{equation}
\|e_h\|_{H^m(\Omega)}
\le
C_{\mathrm{stab}}
\|f-\mathcal{L}u_h\|_{L^2(\Omega)}.
\label{eq:proof_stability}
\end{equation}

Using the sampling equivalence,
\begin{equation}
\|f-\mathcal{L}u_h\|_{L^2(\Omega)}^2
\sim
\sum_{k=1}^{M}
|\mathcal{L}u_h(x_k)-f(x_k)|^2.
\label{eq:proof_sampling_residual}
\end{equation}

Applying the Gram equivalence
\eqref{eq:gram_equivalence_theorem},
\begin{equation}
\sum_{k=1}^{M}
|\mathcal{L}u_h(x_k)-f(x_k)|^2
\le
c_G^{-1}
R(u_h)^TG^{-1}R(u_h).
\label{eq:proof_gram_lower}
\end{equation}

Therefore,
\begin{equation}
\|f-\mathcal{L}u_h\|_{L^2(\Omega)}
\le
C
\|R(u_h)\|_G.
\label{eq:proof_residual_norm}
\end{equation}

Substituting into
\eqref{eq:proof_stability}
gives
\begin{equation}
\|u-u_h\|_{H^m(\Omega)}
\le
C
\|R(u_h)\|_G.
\label{eq:proof_error_estimate}
\end{equation}

Finally, combining
\eqref{eq:proof_error_estimate}
with
\eqref{eq:proof_residual_bound}
yields
\begin{equation}
\|u-u_h\|_{H^m(\Omega)}
\le
Ch^{p+1-s}
|u|_{H^{p+1}(\Omega)}.
\end{equation}

Hence,
$
\|u-u_h\|_{H^m(\Omega)}
=
\mathcal{O}(h^{p+1-s}),
$
which proves
\eqref{eq:final_rate_theorem}.

\end{proof}

\noindent{\bf Acknowledgments} This work has been supported by the National Science Centre, Poland grant no. 
2025/57/B/ST6/00058. 
This work has received funding from the European Union's Horizon Europe research and innovation programme under the Marie Sklodowska-Curie grant agreement No 101119556.
The authors are grateful for the support from the funds that the Polish Ministry of Science and Higher Education assigned to AGH University of Krakow. The work is supported by the “Excellence initiative - research university” for AGH University of Krakow.

\noindent{\bf Declaration of Generative AI and AI-assisted technologies in the writing process}
During the preparation of this work, the author(s) used ChatGPT to assist with writing, correct English language usage, and improve text clarity and organization. After using this tool/service, the author(s) reviewed and edited the content as needed and take(s) full responsibility for the content of the publication.

\end{document}